%% file: MatrixCompletionManifold-ext-report.tex
\documentclass[11pt]{article}

\usepackage{booktabs}
\usepackage{algorithm}
\usepackage{algorithmic}
\usepackage{multirow}
\usepackage{url}
\usepackage{fullpage}

\usepackage{graphicx,color}

\usepackage{amsgen}
\usepackage[cmex10]{amsmath}
\usepackage{amstext}
\usepackage{amsbsy}
\usepackage{amsopn}
\usepackage{amsfonts}               
\usepackage{amssymb}

\PassOptionsToPackage{linktocpage}{hyperref}

\usepackage{verbatim}   
\usepackage{scrtime}    

\usepackage{cite}
\usepackage{array}
\usepackage{mdwmath}
\usepackage{mdwtab}

\usepackage{colonequals}

\newtheorem{thrm}{Theorem}[section]
\newtheorem{lmm}[thrm]{Lemma}

\newtheorem{prpstn}[thrm]{Proposition}

\newtheorem{dfntn}[thrm]{Definition}

\DeclareMathOperator{\rank}{rank}
\DeclareMathOperator{\diag}{diag}
\DeclareMathOperator{\argmin}{arg\ min}
\DeclareMathOperator{\polylog}{polylog}

\DeclareMathOperator{\dom}{dom}
\DeclareMathOperator{\Deriv}{D}
\DeclareMathOperator{\deriv}{d}
\DeclareMathOperator{\Hess}{Hess}
\DeclareMathOperator{\Grad}{grad}
\DeclareMathOperator{\proj}{P}
\DeclareMathOperator{\Tr}{tr}

\newcommand{\Manifold}{\ensuremath{\mathcal{M}_k}}

\newcommand{\TSpace}{\ensuremath{{T}}}
\newcommand{\RSpace}{\ensuremath{\mathbb{R}}}

\newcommand{\ST}{\ensuremath{\mathrm{St}}}

\newcommand{\norm}[1]{\ensuremath{\|#1\|}}

\newcommand{\normF}[1]{\norm{#1}_{\text{F}}}

\newcommand{\innprod}[2]{\ensuremath{\langle #1,\, #2 \rangle}}
\newcommand{\soR}{R^{(2)}}

\newcommand{\projTxM}{\textrm{P}_{\TSpace_X\Manifold}}

\newcommand{\projTxMp}{\textrm{P}_{\TSpace_X\Manifold}^{\textrm{p}}}
\newcommand{\projM}{\textrm{P}_{\Manifold}}

\newcommand{\defeq}{\colonequals}
\newcommand{\calU}{\ensuremath{\mathcal{U}}}
\newcommand{\calH}{\ensuremath{\mathcal{H}}}

\newenvironment{proof}{{\it Proof.}}{\hspace{\stretch{1}} $\square$}

\usepackage{eqparbox}

\title{Low-rank matrix completion by Riemannian optimization---extended version}
\author{Bart Vandereycken\thanks{Chair of Numerical Algorithms and HPC, MATHICSE, \'Ecole Polytechnique F\'ed\'erale de Lausanne, Station 8, CH-1015 Lausanne, Switzerland ({\tt bart.vandereycken@epfl.ch}).} }
\date{September 15, 2012}

\begin{document}
	
	\maketitle

\begin{abstract}
The matrix completion problem consists of finding or approximating a low-rank matrix based on a few samples of this matrix. We propose a new algorithm for matrix completion that minimizes the least-square distance on the sampling set over the Riemannian manifold of fixed-rank matrices. The algorithm is an adaptation of classical non-linear conjugate gradients, developed within the framework of retraction-based optimization on manifolds. We describe all the necessary objects from differential geometry necessary to perform optimization over this low-rank matrix manifold, seen as a submanifold embedded in the space of matrices. In particular, we describe how metric projection can be used as retraction and how vector transport lets us obtain the conjugate search directions. Finally, we prove convergence of a regularized version of our algorithm under the assumption that the restricted isometry property holds for incoherent matrices throughout the iterations. The numerical experiments indicate that our approach scales very well for large-scale problems and compares favorably with the state-of-the-art, while outperforming most existing solvers. 

\emph{This report is the extended version of the manuscript~\cite{Vandereycken:2012d}. It differs only by the addition of Appendix~\ref{app}.}
\end{abstract} 



	
\section{Introduction}\label{sec:intro}

Let $A \in \RSpace^{m \times n}$ be an $m \times n$ matrix that is only known on a subset $\Omega$ of the complete set of entries $\{1, \ldots, m\} \times \{1, \ldots, n\}$. The low-rank matrix completion problem \cite{Candes:2009b} consists of finding the matrix with lowest rank that agrees with $A$ on $\Omega$:
\begin{equation}\label{eq:optim_rank_nonoise}
\begin{array}{ll}
\underset{X}{\text{minimize}}  & \rank(X), \\ 
\text{subject to} & X \in \RSpace^{m \times n}, \ \proj_{\Omega}(X) =  \proj_{\Omega}(A),
\end{array}.
\end{equation}
where
\begin{equation}\label{eq:proj_omega}
 \proj_{\Omega} \colon \RSpace^{m \times n} \to \RSpace^{m \times n}, \  X_{i,j} \mapsto \begin{cases} X_{i,j} & \text{if $(i,j) \in \Omega$}, \\ 0 &  \text{if $(i,j) \not\in \Omega$}, \end{cases}
\end{equation}
denotes the orthogonal projection onto $\Omega$. Without loss of generality, we assume $m \leq n$.

Due to the presence of noise, it is advisable to relax the equality constraint in \eqref{eq:optim_rank_nonoise} to allow for misfit. Then, given a tolerance $\varepsilon \geq 0$, a more robust version of the rank minimization problem becomes
\begin{equation}\label{eq:optim_rank}
\begin{array}{ll}
\underset{X}{\text{minimize}}  & \rank(X), \\ 
\text{subject to} & X \in \RSpace^{m \times n}, \ \norm{ \proj_{\Omega}(X) - \proj_{\Omega}(A) }_{\textrm{F}} \leq \varepsilon,
\end{array}.
\end{equation}
where $\normF{X}$ denotes the Frobenius norm of $X$.

Matrix completion has a number of interesting applications such as collaborative filtering, system identification, and global positioning, but unfortunately it is NP hard. Recently, there has been a considerable body of work devoted to the identification of large classes of matrices for which  \eqref{eq:optim_rank_nonoise} has a unique solution that can be recovered in polynomial time. In \cite{Candes:2009}, for example, the authors show that when $\Omega$ is sampled uniformly at random, the nuclear norm relaxation
\begin{equation}\label{eq:optim_rank_nuclear}
\begin{array}{ll}
\underset{X}{\text{minimize}}  & \norm{X}_*, \\ 
\text{subject to} & X \in \RSpace^{m \times n}, \ \proj_{\Omega}(X) =  \proj_{\Omega}(A),
\end{array}
\end{equation}
can recover with high probability any matrix $A$ of rank $k$ that has so-called low incoherence, provided that the number of samples is large enough, $| \Omega | > C n k \polylog(n)$. Related work has been done by \cite{Candes:2009b, Keshavan:2010, Candes:2010, Keshavan:2010a}, which in particular also establish similar recovery results for the robust formulation \eqref{eq:optim_rank}.

Many of the potential applications for matrix completion involve very large data sets; the Netflix matrix, for example, has more than $10^8$ entries \cite{Netflix:07}. It is therefore crucial to develop algorithms that can cope with such a large-scale setting, but, unfortunately, solving \eqref{eq:optim_rank_nuclear} by off-the-shelf methods for convex optimization scales very badly in the matrix dimension. This has spurred a considerable amount of algorithms that aim to solve the nuclear norm relaxation by specifically designed methods that try to exploit the low-rank structure of the solution; see, e.g., \cite{Cai:2010, Ma:2011, Meka:2010, Lin:2009a, Goldfarb:2011, Toh:2010, Liu:2010}. Other approaches include optimization on the Grassmann manifold \cite{Keshavan:2010, Balzano:2010, Dai:2011}; atomic decompositions \cite{Lee:2010}, and non-linear SOR \cite{Wen:2010}.


\subsection{The proposed method: optimization on manifolds}\label{sec:mat_compl}

We present a new method for low-rank matrix completion based on a direct optimization over the set of all fixed-rank matrices. By prescribing the rank of the global minimizer of \eqref{eq:optim_rank}, say $k$, the robust matrix completion problem is equivalent to
\begin{equation}\label{eq:optim}
\begin{array}{ll}
\underset{X}{\text{minimize}}  & f(X) := \tfrac{1}{2} \norm{ \proj_{\Omega}(X - A) }_{\textrm{F}}^2,  \\ 
\text{subject to} & X \in \mathcal{M}_k := \{ X \in \RSpace^{m \times n} : \rank(X) = k \}.
\end{array}
\end{equation}

It is well known that $\Manifold$ is a smooth ($C^\infty$) manifold; it can, for example, be identified as the smooth part of the determinantal variety of matrices of rank at most $k$ \cite[Proposition~1.1]{Bruns:1988}. Since the objective function $f$ is also smooth, problem \eqref{eq:optim} is a smooth optimization problem, which can be solved by methods from Riemannian optimization as introduced amongst others by \cite{Edelman:1999, Adler:2002, AMS2008}. Simply put, Riemannian optimization is the generalization of standard unconstrained optimization, where the search space is $\RSpace^n$, to optimization of a smooth objective function on a Riemannian manifold. 

For solving the optimization problem \eqref{eq:optim}, in principle any method from optimization on Riemannian manifolds could be used.  In this paper, we use a generalization of classical non-linear conjugate gradients (CG) on Euclidean space to perform optimization on manifolds; see, e.g., \cite{Smith:1994, Edelman:1999, AMS2008}. The reason for this choice compared to, say, Newton's method was that CG performed best in our numerical experiments. The skeleton of the proposed method, LRGeomCG, is listed in Algorithm~\ref{alg:lrgeomcg}.

Algorithm~\ref{alg:lrgeomcg} is derived using concepts from differential geometry, yet it closely resembles a typical non-linear CG algorithm with Armijo line-search for unconstrained optimization. It is schematically visualized in Figure~\ref{fig:optim} for iteration number $i$ and relies on the following crucial ingredients which will be explained in more detail for \eqref{eq:optim} in Section~\ref{sec:riem_optim}.
\begin{enumerate}	
\item The \emph{Riemannian gradient}, denoted $\Grad f(X_i)$, is a specific tangent vector $\xi_i$ which corresponds to the direction of steepest ascent of $f(X_i)$, but restricted to only directions in the tangent space $T_{X_i}\Manifold$.
\item The search direction $\eta_i \in T_{X_i}\Manifold$ is conjugate to the gradient and is computed by a variant of the classical Polak--Ribi\`ere updating rule in non-linear CG. This requires taking a linear combination of the Riemannian gradient with the previous search direction $\eta_{i-1}$. Since $\eta_{i-1}$ does not lie in $T_{X_i}\Manifold$, it needs to be transported to $T_{X_i}\Manifold$. This is done by a mapping $\mathcal{T}_{X_{i-1} \to X_i}: T_{X_{i-1}} \Manifold \to T_{X_{i}} \Manifold$, the so-called \emph{vector transport}.
\item As a tangent vector only gives a direction but not the line search itself on the manifold, a smooth mapping $R_{X_i}: T_{X_i}\Manifold \to \Manifold$, called the \emph{retraction}, is needed to map tangent vectors to the manifold. Using the conjugate direction $\eta_i$, a line-search can then be performed along the curve $t \mapsto R_{X_i}(t\, \eta_i)$.  Step 6 uses a standard backtracking procedure where we have chosen to fix the constants. More judiciously chosen constants can sometimes improve the line search, but our numerical experiments indicate that this is not necessary in the setting under consideration.
\end{enumerate}

\begin{algorithm}
  \caption{LRGeomCG: geometric CG for \eqref{eq:optim}}
\begin{algorithmic}[1]
  \REQUIRE initial iterate $X_1 \in \mathcal{M}_k$, tolerance $\tau > 0$, tangent vector $\eta_0 = 0$ 
\FOR{$i = 1,2,\ldots$}  
\STATE Compute the gradient  \\
 \quad $\xi_i := \Grad f(X_i)$ \COMMENT{see Algorithm~\ref{alg:gradient}} 
\STATE Check convergence \\
 \quad if $\norm{\xi_i} \leq \tau $ then break
\STATE Compute a conjugate direction by PR+   \\
 \quad $\eta_i := - \xi_i + \beta_i \,  \mathcal{T}_{X_{i-1} \to X_{i}}(\eta_{i-1})$    \COMMENT{see Algorithm~\ref{alg:beta_adjust}}
\STATE Determine an initial step $t_i$ from the linearized problem \\
 \quad $t_i = \argmin_t f(X_i + t \,\eta_i)$    \COMMENT{see Algorithm~\ref{alg:initial_guess_line_search}}
\STATE Perform Armijo backtracking to find the smallest integer $m \geq 0$ such that\\
 \quad $f(X_i) - f(R_{X_i}(0.5^m\, t_i\,\eta_i )) \geq - 0.0001 \times 0.5^m\, t_i\, \langle\, \xi_i, \eta_i \,  \rangle$\\
and obtain the new iterate \\
 \quad $X_{i+1} := R_{X_i}(0.5^m \,t_i\,\eta_i )$   \COMMENT{see Algorithm~\ref{alg:retraction}\phantom{\hspace{0.12cm}}}
\ENDFOR
 \end{algorithmic}\label{alg:lrgeomcg}
\end{algorithm}

\begin{figure}
	\vspace{0.3cm}
\centering
\def\svgwidth{0.60\columnwidth}
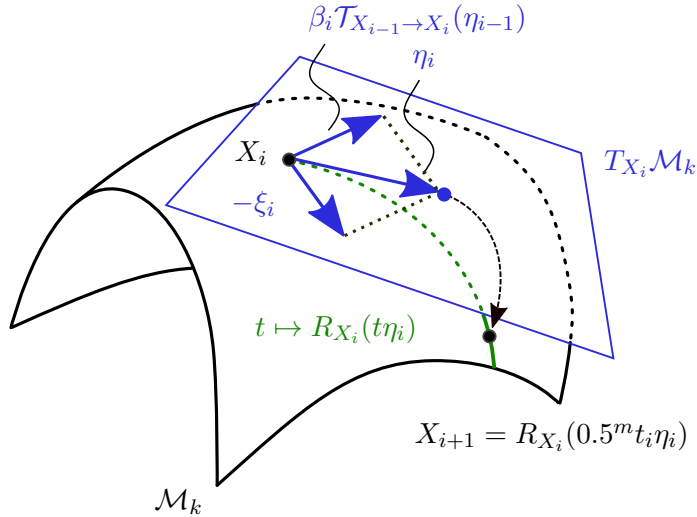
\caption{Visualization of Algorithm~\ref{alg:lrgeomcg}: non-linear CG on a Riemannian manifold.}\label{fig:optim}
\end{figure}

\subsection{Relation to existing manifold-related methods}
 
At the time of submitting the present work, a large number of other matrix completion solvers based on Riemannian optimization have been proposed in \cite{Dai:2010, Keshavan:2010,Keshavan:2010a, Boumal:2011, Mishra:2011b, Meyer:2011c, Shalit:2012, Meyer:2011a, Boumal:2012fk}. Like the current paper, all of these algorithms use the concept of retraction-based optimization on the manifold of fixed-rank matrices but they differ in their specific choice of Riemannian manifold structure and metric. It remains a topic of further investigation to assess the performance of these different geometries with respect to each other and to non-manifold based solvers. 

A first attempt at such a comparison has been very recently done in \cite{Mishra:2012fk} where the previous geometries were tested to complete one large matrix. The authors reach the conclusion that for gradient-based algorithms all these geometries perform remarkably more or less the same with respect to the total computational time. For the particular setup of \cite{Mishra:2012fk}, the $GH^T$ geometry of \cite{Meyer:2011a} turned out to result in the most efficient solver, but overall most geometries---including ours---outperformed the state-of-the art. In addition, Newton-based algorithms performed well when high precision was required and now the algorithm based on our embedded submanifold geometry was clearly faster.

\subsection{Outline of the paper}

The plan of the paper is as follows. In the next section, the necessary concepts of differential geometry are explained to turn Algorithm~\ref{alg:lrgeomcg} into a concrete method. The implementation of each step is explained in Section~\ref{sec:impl_details}. We also prove convergence of a slightly modified version of this method in Section~\ref{sec:convergence} under some assumptions which are reasonable for matrix completion. Numerical experiments and comparisons to the state-of-the art are carried out in Section~\ref{sec:num_exp}. The last section is devoted to conclusions.

\section{Differential geometry for low-rank matrix manifolds} \label{sec:riem_optim}

In this section, we explain the differential geometry concepts used in Algorithm~\ref{alg:lrgeomcg} applied to our particular matrix completion problem \eqref{eq:optim}.

\subsection{The Riemannian manifold}

Let
\[
 \Manifold = \{ X \in \RSpace^{m \times n} \colon \rank(X) = k \}
\]
denote the manifold of fixed-rank matrices. Using the SVD, one has the equivalent characterization
\begin{equation}\label{eq:x_USV}
 \Manifold = \{  U \Sigma V^T \colon U \in \ST_k^m,\,  V \in \ST_k^n,\, \Sigma  = \diag(\sigma_i),\, \sigma_1 \geq \cdots \geq \sigma_k > 0 \},
\end{equation}
where $\ST_k^m$ is the Stiefel manifold of $m \times k$ real, orthonormal matrices, and $\diag(\sigma_i)$ denotes a diagonal matrix with $\sigma_i$ on the diagonal. Whenever we use the notation $U \Sigma V^T$ in the rest of the paper, we always mean matrices that satisfy \eqref{eq:x_USV}. Furthermore, the constants $k,m,n$ are always used to denote dimensions of matrices and, for simplicity, we assume $1 \leq k < m \leq n$.

The following proposition shows that $\Manifold$ is indeed a smooth manifold. While the existence of such a smooth manifold structure, together with its tangent space, is well known (see, e.g., \cite{HM94,Koch:2007} for applications of gradient flows on $\Manifold$), more advanced concepts like retraction-based optimization on $\Manifold$ have only very recently been investigated; see \cite{Shalit:2010,Shalit:2012}. In contrast, the case of optimization on \emph{symmetric} fixed-rank matrices has been studied in more detail in \cite{Orsi:2006,Journee:2010,Vandereycken2010,Meyer:2011}.

\begin{prpstn}
 The set $\Manifold$ is a smooth submanifold of dimension $(m+n-k)k$ embedded in $\RSpace^{m \times n}$. Its tangent space $\TSpace_X\Manifold$ at $X=U \Sigma V^T \in \Manifold$ is given by
  \begin{align}
   \TSpace_X\Manifold &= \left\{ \begin{bmatrix} U & U_\perp \end{bmatrix} \begin{bmatrix} \RSpace^{k \times k} & \RSpace^{k \times (n-k)} \\ \RSpace^{(m-k) \times k} & 0_{(m-k) \times (n-k)} \end{bmatrix}  \begin{bmatrix} V & V_\perp \end{bmatrix}^T  \right\} \label{eq:Tspace_full}\\
   &= \{ UMV^T + U_p V^T + UV_p^T \colon M \in \RSpace^{k \times k}, \label{eq:Tspace_short}\\
  &\qquad\qquad\qquad\qquad U_p \in \RSpace^{m \times k}, \,U_p^TU= 0, \,V_p \in \RSpace^{n \times k}, \,V_p^TV= 0 \}.\nonumber
  \end{align}
\end{prpstn}
\begin{proof}
	See \cite[Example 8.14]{Lee2003} for a proof that uses only elementary differential geometry based on local submersions. The tangent space is obtained by the first-order perturbation of the SVD and counting dimensions.
\qquad\end{proof}		

The \emph{tangent bundle} is defined as the disjoint union of all tangent spaces,
\[
 \TSpace\Manifold := \bigcup_{X \in \Manifold} \{ X \} \times \TSpace_X\Manifold = \{ (X,\xi) \in \RSpace^{m \times n} \times \RSpace^{m \times n} \colon  X \in \Manifold, \,\xi \in \TSpace_X \Manifold \}.
\]
By restricting the Euclidean inner product on $\RSpace^{m \times n}$,
\[
\innprod{A}{B} = \Tr(A^T B) \quad \text{with $A,B \in \RSpace^{m \times n}$},
\]
to the tangent bundle, we turn $\Manifold$ into a Riemannian manifold with Riemannian metric
\begin{equation}\label{eq:metric}
 g_X(\xi, \eta) := \innprod{\xi}{\eta} = \Tr(\xi^T \eta) \quad \text{with $X \in \Manifold$ and $\xi,\eta \in \TSpace_X\Manifold$},
\end{equation}
where the tangent vectors $\xi,\eta$ are seen as matrices in $\RSpace^{m \times n}$.

Once the metric is fixed, the notion of the gradient of an objective function can be introduced. For a Riemannian manifold, the \emph{Riemannian gradient} of a smooth function $f:\Manifold \to \RSpace$ at $X \in \Manifold$ is defined as the unique tangent vector $\Grad f(X)$ in $T_X \Manifold$ such that
\[
 \langle\, \Grad f(X), \xi \, \rangle = \Deriv f(X)[ \xi ] \quad \text{for all $\xi \in T_X\Manifold$} ,
\]
where we denoted directional derivatives by $\Deriv f$. Since $\Manifold$ is embedded in $\RSpace^{m \times n}$, the Riemannian gradient is given as the orthogonal projection onto the tangent space of the gradient of $f$ seen as a function on $\RSpace^{m \times n}$; see, e.g., \cite[(3.37)]{AMS2008}. 
Defining $\proj_{U}:=UU^T$ and  $\proj_{U}^\perp:=I - \proj_{U}$ for any $U \in \ST_k^m$, we denote the orthogonal projection onto the tangent space at $X$ as
\begin{equation}\label{eq:P_X}
 \proj_{T_X \Manifold}: \RSpace^{m \times n} \to T_X\Manifold, \, Z \mapsto \proj_U Z \proj_V + \proj_U^\perp Z \proj_V + \proj_U Z \proj_V^\perp.
\end{equation}
Then, using $\proj_{\Omega}(X - A)$ as the (Euclidean) gradient of $f(X) = \norm{ \proj_{\Omega}(X - A) }_{\textrm{F}}^2 / 2$, we obtain
\begin{equation}\label{eq:riem_gradient}
 \Grad f(X) := \proj_{T_X \Manifold} (\proj_\Omega( X-A )).
\end{equation}

\subsection{Metric projection as retraction}

As explained in Section~\ref{sec:mat_compl}, we need a so-called retraction mapping to go back from an element in the tangent space to the manifold. In order to prove convergence within the framework of \cite{AMS2008}, this mapping has to satisfy certain properties such that it becomes a first-order approximation of the exponential mapping on $\Manifold$.

\begin{dfntn}[Retraction \protect{\cite[Definition~1]{Absil:2010}}]\label{def:retr}
  A mapping $R \colon \TSpace \Manifold \to \Manifold$ is said to be a \emph{retraction on $\Manifold$} if, for every $\overline{X} \in \Manifold$, there exists a neighborhood $\calU$ around $(\overline{X},0) \subset \TSpace\Manifold$ such that the following holds.
  \begin{enumerate}
    \item[(1)] $\calU \subseteq \dom(R)$ and $R\vert_{\calU}\colon \calU \to \Manifold$ is smooth.
    \item[(2)] $R(X,0) = X$ for all $(X,0) \in \calU$.
    \item[(3)] $\Deriv R(X,0)[0,\xi] = \xi$ for all $(X,\xi) \in \calU$.
  \end{enumerate}  
\end{dfntn}
We also introduce the following shorthand notation: 
\[
 R_X: \TSpace_X\Manifold \to \Manifold, \ \xi \mapsto R(X,\xi).
\]

In our setting, we have chosen metric projection as retraction:
\begin{equation}\label{eq:R_X}
 R_X:  \mathcal{U}_X \to \Manifold, \ \xi \mapsto \projM( X+\xi ) := \underset{Z \in \Manifold}{\argmin}\, \norm{X+\xi - Z}_{\text{F}}, 
\end{equation}
where $\mathcal{U}_X \subset T_X\Manifold$ is a suitable neighborhood around the zero vector, and $\projM$ is the orthogonal projection onto $\Manifold$. Based on \cite[Lemma~2.1]{Lewis2008}, it can be easily shown that this map satisfies the conditions in Definition~\ref{def:retr}.

The retraction as metric projection can be computed in closed-form by the SVD: Let $X \in \Manifold$ be given, then for sufficiently small $\xi \in \mathcal{U}_X \subset \TSpace_X \Manifold$, we have
\begin{equation}\label{eq:Rx_SVD}
 R_X(\xi) =  \projM( X+\xi ) = \sum_{i=1}^k \sigma_i u_i  v_i^T,
\end{equation}
where $\sigma_i, u_i, v_i$ are the (ordered) singular values and vectors of the SVD of $X + \xi$. It is obvious that when $\sigma_{k}=0$ there is no best rank-$k$ solution for \eqref{eq:R_X}, and additionally, when $\sigma_{k-1} = \sigma_{k}$ the minimization in \eqref{eq:R_X} does not have a unique solution. 
In fact, since the manifold is non-convex, metric projections can never be well defined on the whole tangent bundle. Fortunately, as indicated in Definition~\ref{def:retr}, the retraction has to be
defined only locally because this is sufficient to establish convergence of the Riemannian algorithms.

\subsection{Riemannian Newton on $\Manifold$}

Although we do not exploit second-order information in Algorithm~\ref{alg:lrgeomcg}, Newton or its variants may be preferable for some applications. In \cite{Mishra:2012fk}, for example, it was shown that the second-order Riemannian trust-region algorithm of \cite{Absil:2007} performs very well in combination with our embedded submanifold geometry. We will therefore give the following theorem with an explicit expression of the Riemannian Hessian of $f$. Its proof is a straightforward but technical analog of that in \cite{Vandereycken2010} for symmetric fixed-rank matrices. It can be found in Appendix~\ref{app}.

\begin{prpstn}\label{prop:Hessian}
For any $X=U \Sigma V^T \in \Manifold$ satisfying \eqref{eq:x_USV} and $\xi \in T_X\Manifold$ satisfying \eqref{eq:Tspace_short}, the Riemannian Hessian of $f$ at $X$ in the direction of $\xi$ satisfies
\begin{align*}
 \Hess f(X)[\xi] &= \proj_U \proj_\Omega(\xi) \proj_V  + \proj_U^\perp \left[\proj_\Omega(\xi) + \proj_\Omega(X-A) V_p \Sigma^{-1} V^T \right]  \proj_V \\
 &\quad + \proj_U (\proj_\Omega(\xi) + U \Sigma^{-1}U_p^T\proj_\Omega(X-A))\proj_V^\perp.
\end{align*}
\end{prpstn}

\subsection{Vector transport}

Vector transport was introduced in \cite{AMS2008} and \cite{Qi:2010} as a means to transport tangent vectors from one tangent space to another. In a similar way as retractions are approximations of the exponential mapping, vector transport is the first-order approximation of parallel transport, another important concept in differential geometry.

The definition below makes use of the Whitney sum, which is the vector bundle over some space obtained as the direct sum of two vector bundles over that same space. In our setting, we can define it as
\begin{align*}
 T\Manifold \oplus T\Manifold &= \bigcup_{X \in \Manifold} \{ X \} \times \TSpace_X\Manifold \times \TSpace_X\Manifold \\
  &= \{ (X,\eta,\xi) \colon  X \in \Manifold, \xi \in \TSpace_X \Manifold, \eta \in \TSpace_X \Manifold  \}.
\end{align*}

\begin{dfntn}[Vector transport \protect{\cite[Definition 2]{Qi:2010}}]\label{def:vec_transp}
	A vector transport on the manifold $\Manifold$ is a smooth mapping: $T\Manifold \oplus T\Manifold \to T\Manifold, (X,\eta,\xi) \mapsto \mathcal{T}_{\eta}(\xi)$ satisfying the following properties for all $X \in \Manifold$.
  \begin{enumerate}
    \item[(1)] There exists a retraction $R$ such that, for all $(X,\eta, \xi) \in T\Manifold \oplus T\Manifold$, it holds that $\mathcal{T}_{\eta}(\xi) \in T_{R_X(\eta)} \Manifold$.
    \item[(2)] $\mathcal{T}_{0}(\xi) = \xi$ for all $\xi \in T_X\Manifold$.
    \item[(3)] For all $(X,\eta) \in T\Manifold$, the mapping $\mathcal{T}_{\eta}: T_X\Manifold \to T_{R_X(\eta)}\Manifold, \, \xi \mapsto \mathcal{T}_{\eta}(\xi)$ is linear.
  \end{enumerate}  
\end{dfntn}

By slight abuse of notation, we will use the following shorthand notation that emphasizes the two tangent spaces involved:
\[
 \mathcal{T}_{X \to Y}: T_X \Manifold \to T_Y \Manifold, \, \xi \mapsto \mathcal{T}_{R_X^{-1}(Y)}(\xi),
\]
where $\mathcal{T}_{R_X^{-1}(Y)}$ uses the notation of Definition~\ref{def:vec_transp}. See also Figure~\ref{fig:vec_transp} for an illustration of this definition. By the implicit function theorem, $R_X$ is locally invertible for every $X \in \Manifold$ but $R_X^{-1}(Y)$ does not need to exist for all $Y \in \Manifold$. Hence, $\mathcal{T}_{X \to Y}$ has to be understood locally, i.e., for $Y$ sufficiently close to $X$.

\begin{figure}
	\vspace{0.3cm}
\centering
\def\svgwidth{0.65\columnwidth}
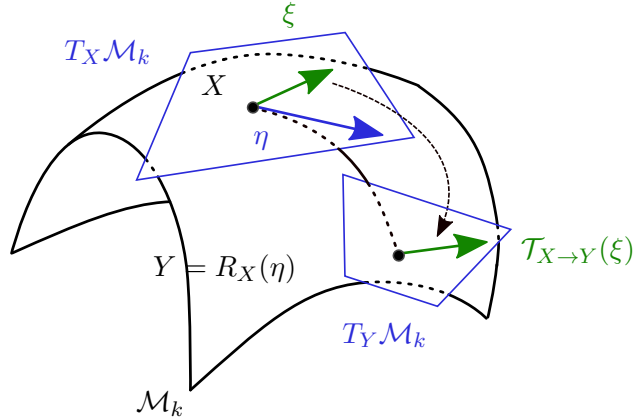
\caption{Vector transport on a Riemannian manifold.}\label{fig:vec_transp}
\end{figure}

Since $\Manifold$ is an embedded submanifold of $\RSpace^{m \times n}$, orthogonally projecting the translated tangent vector in $\RSpace^{m \times n}$ onto the new tangent space constitutes a vector transport; see \cite[Section 8.1.2]{AMS2008}. In other words, we get
\begin{equation}\label{eq:vector_transport}
 \mathcal{T}_{X \to Y}: T_X \Manifold \to T_Y \Manifold, \, \xi \mapsto \proj_{T_Y\Manifold}(\xi),
\end{equation}
with $\proj_{T_Y\Manifold}$ as defined in \eqref{eq:P_X}.

As explained in Section~\ref{sec:mat_compl}, the conjugate search direction $\eta_i$ in Algorithm~\ref{alg:lrgeomcg} is computed as a linear combination of the gradient and the previous direction:
\[
 \eta_i = - \Grad f(X_i) + \beta_i \,  \mathcal{T}_{X_{i-1} \to X_{i}}(\eta_{i-1}),
\]
where we transported  $\eta_{i-1}$. For $\beta_i$, we have chosen the geometrical variant of Polak--Ribi\`ere (PR+), as introduced in \cite[Chapter 8.2]{AMS2008}. Again using vector transport, this becomes
\begin{equation}
	\beta_i = \frac{ \langle \, \Grad f(X_i), \  \Grad f(X_i) - \mathcal{T}_{X_{i-1}\to X_i}(\Grad f(X_{i-1})) \, \rangle }{ \langle\ \Grad f(X_{i-1}),  \Grad f(X_{i-1}) \, \rangle }.
\end{equation}
In order to prove convergence, we also enforce that the search direction $\eta_i$ is sufficiently gradient-related in the sense that its angle with the gradient is never too small; see, e.g., \cite[Chap.~1.2]{Bertsekas:1999}.


\section{Implementation details}\label{sec:impl_details}

This section is devoted to the implementation of Algorithm~\ref{alg:lrgeomcg}. For most operations, we will also provide a flop count for the low-rank regime, i.e., $k \ll m  \leq n$.

\begin{paragraph}{Low-rank matrices}
Since every element $X \in \mathcal{M}_k$ is a rank $k$ matrix, we store it as the result of a compact SVD:    
\[
 X = U \Sigma V^T,
\]
with orthonormal matrices $U \in \ST_k^m$ and $V \in \ST_k^n$, and a diagonal matrix $\Sigma \in \RSpace^{k \times k}$ with decreasing positive entries on the diagonal.
For the computation of the objective function and the gradient, we also precompute the sparse matrix $X_{\Omega} := \proj_{\Omega}(X)$. As explained in the next paragraph, computing $X_{\Omega}$ costs $(m + 2| \Omega |) k$ flops.
\end{paragraph}

\begin{paragraph}{Projection operator $\proj_{\Omega}$} 
During the course of Algorithm~\ref{alg:lrgeomcg}, we require the application of $\proj_{\Omega}$, as defined in \eqref{eq:proj_omega}, to certain low-rank matrices. By exploiting the low-rank form, this can be done efficiently as follows: Let $Z$ be a rank-$k_Z$ matrix with factorization $Z:=Y_1 Y_2^T$ ($k_Z$ is possibly different from $k$, the rank of the matrices in $\mathcal{M}_k$). Define $Z_{\Omega}:=\proj_{\Omega}(Z)$. Then element $(i,j)$ of $Z_{\Omega}$ is given by
\begin{equation}
(Z_{\Omega})_{i,j} =
\begin{cases}
 \sum_{l=1}^{k_Z} Y_1(i,l) Y_2(j,l) & \text{if $(i,j) \in \Omega$}, \\
 0 & \text{if $(i,j) \not\in \Omega$}.
\end{cases}  \label{eq:proj_omega_low_rank}
\end{equation}
The cost for computing $Z_{\Omega}$ is $2 | \Omega | k_Z$ flops.
\end{paragraph}

\begin{paragraph}{Tangent vectors} A tangent vector $\eta \in T_X \mathcal{M}_k$ at $X=U\Sigma V^T \in \mathcal{M}_k$ will be represented as
\begin{equation}\label{eq:tangent_vectors}
 \eta = U M V^T + U_p V^T + UV_p^T,
\end{equation}
where $M \in \RSpace^{k \times k}$, $U_p \in \RSpace^{m \times k}$ with $U^T U_p = 0$, and $V_p \in \RSpace^{n \times k}$ with $V^T V_p = 0$. After vectorizing this representation, the inner product $\langle \, \eta, \nu \, \rangle$ can be computed in $2(m+n)k + 2k^2$ flops.
 
Since $X$ is available as an SVD, the orthogonal projection onto the tangent space $T_X\mathcal{M}_k$ becomes
\[
  \proj_{T_X \Manifold}(Z) := \proj_U Z \proj_V + (Z-\proj_U Z) \proj_V + \proj_U (Z^T-\proj_V Z^T)^T,
\]
where $\proj_U := UU^T$ and $\proj_V:=VV^T$. If $Z$ can be efficiently applied to a vector, evaluating and storing the result of $\proj_{T_X \Manifold}(Z)$ as a tangent vector in the format \eqref{eq:tangent_vectors} can also be performed efficiently.
\end{paragraph}

\begin{paragraph}{Riemannian gradient} From \eqref{eq:riem_gradient}, the Riemannian gradient at $X=U \Sigma V^T$ is the orthogonal projection of $\proj_{\Omega}(X-A)$ onto the tangent space at $X$. Since the sparse matrix $A_{\Omega}= \proj_{\Omega}(A)$ is given and $X_{\Omega} = \proj_{\Omega}(X)$ was already precomputed, the gradient can be computed by Algorithm~\ref{alg:gradient}. The total cost is $2(n+2m)k^2+4|\Omega|k$ flops.

\begin{algorithm}
  \caption{Calculate gradient $\Grad f(X)$}\label{alg:gradient}
\begin{algorithmic}[1]
  \REQUIRE matrix $X=U\Sigma V^T \in \mathcal{M}_k$, sparse matrix $R = X_{\Omega} - A_{\Omega} \in \RSpace^{m \times n}$.
  \ENSURE $\Grad f(X) = U M V^T + U_p V^T + UV_p^T \in T_X\mathcal{M}_k$
\STATE $R_u \leftarrow R^T U$, \quad $R_v \leftarrow R V $  \COMMENT{$4|\Omega|k$ flops}
\STATE $M \leftarrow U^T R_v $ \COMMENT{$2mk^2$ flops}
\STATE $U_p \leftarrow R_v - U M$, \quad $V_p \leftarrow R_u - V M^T$ \COMMENT{$2(m+n)k^2$ flops\phantom{\hspace{0.12cm}}}
\end{algorithmic}
\end{algorithm}
\end{paragraph}

\begin{paragraph}{Vector transport} Let $\nu = U M V^T + U_p V^T + U V_p^T$ be a tangent vector at $X=U \Sigma V^T \in \mathcal{M}_k$. By \eqref{eq:vector_transport}, the vector
\[
\nu_+ := \mathcal{T}_{X \to X_+}(\nu) = \proj_{T_{X_+}\Manifold}(\nu)
\] 
is the transport of $\nu$ to the tangent space at some $X_+ = U_+ \Sigma_+ V_+^T$. It is computed by Algorithm~\ref{alg:transp_vector} with a total cost of approximately $14(m+n)k^2 + 10k^3$ flops. 

\begin{algorithm}
  \caption{Calculate vector transport $\mathcal{T}_{X \to X_+}(\nu)$}\label{alg:transp_vector}
\begin{algorithmic}[1]
  \REQUIRE matrices $X=U\Sigma V^T \in \mathcal{M}_k$ and  $X_+=U_+\Sigma_+ V_+^T \in \mathcal{M}_k$, \\tangent vector $\nu = U M V^T + U_p V^T + U V_p^T$.
  \ENSURE $\mathcal{T}_{X \to X_+}(\nu)= U_+ M_+ V_+^T + U_{p_+} V_+^T + U_+V_{p_+}^T \in T_{X_+}\mathcal{M}_k$
  \STATE $A_v \leftarrow V^TV_+, \quad A_u \leftarrow U^TU_+$    \COMMENT{$2(m+n)k^2$}
  \STATE $B_v \leftarrow V_p^TV_+, \quad B_u \leftarrow U_p^TU_+$    \COMMENT{$2(m+n)k^2$}
  \STATE $M_+^{(1)} \leftarrow A_u^T M A_v, \quad U_{+}^{(1)} \leftarrow U (M A_v), \quad V_{+}^{(1)} \leftarrow V (M^T A_u)$ \COMMENT{$6k^3 + 2(m+n)k^2$}
  \STATE $M_+^{(2)} \leftarrow B_u^T A_v, \quad U_{+}^{(2)} \leftarrow U_p A_v, \quad V_{+}^{(2)} \leftarrow V B_u $  \COMMENT{$2k^3 + 2(m+n)k^2$}
  \STATE $M_+^{(3)} \leftarrow A_u^T B_v, \quad U_{+}^{(3)} \leftarrow U B_v, \quad V_{+}^{(3)} \leftarrow V_p A_u$ \COMMENT{$2k^3 + 2(m+n)k^2$}
  \STATE $M_+ \leftarrow M_+^{(1)} + M_+^{(2)} + M_+^{(3)}$   \COMMENT{$2k^2$}
  \STATE $U_{p_+} \leftarrow U_{+}^{(1)} + U_{+}^{(2)} + U_{+}^{(3)}$, \quad $U_{p_+} \leftarrow U_{p_+} - U_+(U_+^T U_{p_+})$  \COMMENT{$4mk^2$}
  \STATE $V_{p_+} \leftarrow V_{+}^{(1)} + V_{+}^{(2)} + V_{+}^{(3)}$, \quad $V_{p_+} \leftarrow V_{p_+} - V_+(V_+^T V_{p_+})$  \COMMENT{$4nk^2$\phantom{\hspace{0.12cm}}}
\end{algorithmic}
\end{algorithm}
\end{paragraph}

\begin{paragraph}{Non-linear CG}
	The geometric variant of Polak--Ribi\`ere is implemented as Algorithm~\ref{alg:beta_adjust} costing about $28(m+n)k^2 + 20k^3$ flops. In order to improve robustness, we use the PR+ variant and restart when the conjugate direction is almost orthogonal to the gradient. 
	
\begin{algorithm}
  \caption{Compute the conjugate direction by PR+}
\begin{algorithmic}
  \REQUIRE previous iterate $X_{i-1}$, previous gradient $\xi_{i-1}$, previous direction $\eta_{i-1}$\\
  current iterate $X_i$, current gradient $\xi_{i}$
  \ENSURE conjugate direction $\eta_i \in T_{X_i}\mathcal{M}_k$
\STATE {\sc\small 1:} Transport previous gradient and direction to current tangent space:\\
  \qquad $\overline{\xi}_{i} \leftarrow \mathcal{T}_{X_{i-1} \to X_i}(\xi_{i-1})$    \COMMENT{apply Algorithm~\ref{alg:transp_vector}\phantom{\hspace{0.12cm}}} \\
  \qquad $\overline{\eta}_{i} \leftarrow \mathcal{T}_{X_{i-1} \to X_i}(\eta_{i-1})$  \COMMENT{apply Algorithm~\ref{alg:transp_vector}}
\STATE {\sc\small 2:} Compute conjugate direction:\\
  \qquad $\delta_i \leftarrow \xi_i - \overline{\xi}_i $  \\
  \qquad $\beta \leftarrow \max( 0, \langle \delta_i, \xi_i \rangle / \langle \xi_{i-1}, \xi_{i-1} \rangle )$ \\
  \qquad $\eta_i \leftarrow  - \xi_i + \beta\, \overline{\eta}_i$
\STATE {\sc\small 3:} Compute angle between conjugate direction and gradient:\\
  \qquad $\alpha \leftarrow \langle \eta_i, \xi_i \rangle / \sqrt{ \langle \eta_i, \eta_i \rangle \ \langle \xi_i, \xi_i \rangle }$
\STATE {\sc\small 4:} Reset to gradient if desired:\\
  \qquad if $\alpha \leq 0.1$ then $\eta_i \leftarrow \xi_i$
\end{algorithmic}\label{alg:beta_adjust}
\end{algorithm}

\end{paragraph}

\begin{paragraph}{Initial guess for line search}

Although Algorithm~\ref{alg:lrgeomcg} uses line search, a good initial guess can greatly enhance performance. We observed in our numerical experiments that an exact minimization on the tangent space alone (so, neglecting the retraction), 
\begin{equation}\label{eq:initial_guess_Omega}
 \min_t f(X + t \,\eta) = \frac{1}{2} \min_t \norm{ \proj_{\Omega}(X) + t \proj_{\Omega}(\eta) - \proj_{\Omega}(A) }_{\text{F}}^2,
\end{equation}
performed extremely well as initial guess in the sense that backtracking is almost never necessary.

Equation \eqref{eq:initial_guess_Omega} is essentially a one-dimensional least-square fit for $t$ on $\Omega$. The closed-form solution of the minimizer $t_*$ satisfies
\[
 t_* = \langle\, \proj_{\Omega}(\eta), \proj_{\Omega}(A-X) \, \rangle \, / \, \langle\, \proj_{\Omega}(\eta), \proj_{\Omega}(\eta) \, \rangle
\]
and is unique when $\eta \neq 0$. As long as Algorithm~\ref{alg:lrgeomcg} has not converged, $\eta$ will always be non-zero since it is the direction of search. The solution to \eqref{eq:initial_guess_Omega} is performed by Algorithm~\ref{alg:initial_guess_line_search}. In the actual implementation, the sparse matrices $N$ and $B$ are stored as the non-zero entries on $\Omega$. Hence, the total cost is about $2mk^2 + 4|\Omega|(k+1)$ flops.

\begin{algorithm}
  \caption{Compute the initial guess for line search $t_* = \argmin_t f(X + t \,\eta)$}
\begin{algorithmic}
\REQUIRE iterate $X=U \Sigma V^T$ and projection $X_{\Omega}$, tangent vector $\eta=UMV^T + U_p V^T + U V_p^T$, sparse matrix $R = A_{\Omega} - X_{\Omega} \in \RSpace^{m \times n}$
  \ENSURE step length $t_*$
  \STATE  {\sc\small 1:} $N \leftarrow \proj_{\Omega}(\begin{bmatrix} UM+U_p  & U \end{bmatrix} \begin{bmatrix} V & V_p  \end{bmatrix}^T)$   \COMMENT{$2mk^2 + 4 |\Omega |k$ flops}
  \STATE  {\sc\small 2:} $t_*  \leftarrow  \Tr(N^T R) / \Tr(N^T N)$   \COMMENT{$4|\Omega|$ flops\phantom{\hspace{0.12cm}}}
\end{algorithmic}\label{alg:initial_guess_line_search} 
\end{algorithm}

\end{paragraph}

\begin{paragraph}{Retraction} As shown in \eqref{eq:Rx_SVD}, the retraction \eqref{eq:R_X} can be directly computed by the SVD of $X + \xi$. A full SVD of $X + \xi$ is, however, prohibitively expensive since it costs $O(n^3)$ flops. Fortunately, the matrix to retract has the particular form
\[
 X + \xi = \begin{bmatrix} U & U_p \end{bmatrix} \begin{bmatrix} \Sigma + M & I_k \\ I_k & 0 \end{bmatrix} \begin{bmatrix} V & V_p \end{bmatrix}^T,
\]
with $X=U \Sigma V^T \in \Manifold$ and $\xi=UMV^T + U_p V^T + U V_p^T \in T_X\Manifold$. Algorithm~\ref{alg:retraction} now performs a compact QR and an SVD of a small $2k$-by-$2k$ matrix to reduce the flop count to $14(m+n)k^2 + C_{\text{SVD}}k^3$ when $k \ll \min(m,n)$.

Observe that the listing of Algorithm~\ref{alg:retraction} uses {\sc Matlab} notation to denote common matrix operations. In addition, the flop count of computing the SVD in Step 3 is given as $C_{\text{SVD}} k^3$ since it is difficult to estimate beforehand in general In practice, the constant $C_{\text{SVD}}$ is modest, say, less than 200. Furthermore, in Step 4, we have added $\varepsilon_{\text{mach}}$ to $\Sigma_s$ so that, in the unlucky event of zero singular values, the retracted matrix is perturbed to a rank-$k$ matrix in $\Manifold$.

\begin{algorithm}
  \caption{Compute the retraction by metric projection}
\begin{algorithmic}
  \REQUIRE iterate $X=U \Sigma V^T$, tangent vector $\xi=UMV^T + U_p V^T + U V_p^T$
  \ENSURE retraction $R_X(\xi) = \proj_{\Manifold}(X+\xi) = U_+ \Sigma_+ V_+^T$
  \STATE  {\sc\small 1:}  $(Q_u, R_u) \leftarrow \text{qr}(U_p, 0)$, \quad $(Q_v, R_v) \leftarrow \text{qr}(V_p, 0)$  \COMMENT{$10(m+n)k^2$ flops}
  \STATE  {\sc\small 2:} $S \leftarrow \begin{bmatrix} \Sigma + M & R_v^T \\ R_u & 0 \end{bmatrix}$ 
  \STATE  {\sc\small 3:}  $(U_s,\Sigma_s,V_s) \leftarrow \text{svd}(S)$  \COMMENT{$C_{\text{SVD}} k^3$ flops}
  \STATE  {\sc\small 4:}  $\Sigma_+ \leftarrow \Sigma_s(1:k,1:k) + \varepsilon_{\text{mach}}$
  \STATE  {\sc\small 5:}  $U_+ \leftarrow \begin{bmatrix} U & Q_u \end{bmatrix} U_s(\,:\,,1:k)$, \quad $V_+ \leftarrow \begin{bmatrix} V & Q_v \end{bmatrix} V_s(\,:\,,1:k)$ \COMMENT{$4(m+n)k^2$ flops\phantom{\hspace{0.12cm}}}
\end{algorithmic}\label{alg:retraction}
\end{algorithm}

\end{paragraph}

\begin{paragraph}{Computational cost}
	
Summing all the flop counts for Algorithm~\ref{alg:lrgeomcg}, we arrive at a cost per iteration of approximately
\[
 (48m+44n)k^2 + 10 | \Omega | k + \text{nb}_{\text{Armijo}}\, (8(m+n)k^2 + 2| \Omega | k),
\]
where $\text{nb}_{\text{Armijo}}$ denotes the average number of Armijo backtracking steps. It is remarkable, that in all experiments below we have observed that $\text{nb}_{\text{Armijo}}=0$, that is, backtracking was never needed.

For our typical problems in Section~\ref{sec:num_exp}, the size of the sampling set satisfies $| \Omega | = \text{OS}\,(m+n-k)k$ with $ \text{OS} >2$ the oversampling factor. When $m=n$ and $\text{nb}_{\text{Armijo}}=0$, this brings the total flops per iteration to
\begin{equation}\label{eq:total_flops_theory}
 t_\text{theoretical} = (92 + 20 \,\text{OS}) nk^2.
\end{equation}
Comparing this theoretical flop count with experimental results, we observe that sparse matvecs and applying $\proj_{\Omega}$ require significantly more time than predicted by \eqref{eq:total_flops_theory}. This can be explained due to the lack of data locality for these sparse matrix operations, whereas the majority of the remaining time is spent by dense linear algebra, rich in BLAS3. After some experimentation, we estimated that for our {\sc Matlab} environment these sparse operations are penalized by a factor of about $C_\text{sparse} \simeq 5$. For this reason, we normalize the theoretical estimate \eqref{eq:total_flops_theory} to obtain the following practical estimate:
\begin{equation}\label{eq:total_flops_practice}
  t_\text{practical}^{\text{LRGeomCG}} =  ( 92  +  C_\text{sparse}\, 20\, \text{OS} ) nk^2, \quad C_\text{sparse} \simeq 5.
\end{equation}
For a sensible value of $\text{OS}=3$, we can expect that about $75$ percent of the computational time will be spent on operations with sparse matrices.
\end{paragraph}

\begin{paragraph}{Comparison to existing methods}

The vast majority of specialized solvers for matrix completion require computing the dominant singular vectors of a sparse matrix in each step of the algorithm. This is, for example, the case for approaches using soft-  and hard-thresholding, like in \cite{Cai:2010, Ma:2011, Meka:2010, Lin:2009a, Goldfarb:2011, Toh:2010, Liu:2010, Boumal:2011, Mishra:2011b}. Typically, PROPACK from \cite{Larsen:2004} is used for computing such a truncated SVD. The use of sparse matrix-vector products and the potential convergence issues frequently makes this the computationally most expensive part of all these algorithms.

On the other hand, our algorithm first projects any sparse matrix onto the tangent space and then computes the dominant singular vectors for a small $2k$-by-$2k$ matrix. Apart from the application of $\proj_{\Omega}$ and \emph{a few} sparse matrix-vector products, the rest of the algorithm solely consists of dense linear algebra operations. This is more robust and significantly faster than computing the SVD of a sparse matrix in each step. To the best of our knowledge, LMAFit \cite{Wen:2010} and other manifold-related algorithms like \cite{Balzano:2010, Dai:2011, Keshavan:2010, Keshavan:2010a, Meyer:2011} are the only competitive solvers where mostly dense linear algebra together with a few sparse matrix-vector products are sufficient.

Due to the difficulty of estimating practical flop counts for algorithms based on computing a sparse SVD, we will only compare the computational complexity of our method with that of LMAFit. An estimate similar to the one of above, reveals that LMAFit has a computational cost of
\begin{equation}\label{eq:total_flops_lma}
 t_\text{practical}^{\text{LMAFit}} =  ( 18  +  C_\text{sparse}\, 12\, \text{OS} ) nk^2, \quad C_\text{sparse} \simeq 5,
\end{equation}
where the sparse manipulations involve $2k$ sparse matvecs and one application of $\proj_{\Omega}$ to a rank $k$ matrix.

Comparing this estimate to \eqref{eq:total_flops_practice}, LMAFit should be about two times faster per iteration when $\text{OS}=3$. As we will see later in Tables~\ref{tab:conv_rank_known} and~\ref{tab:conv_rank_known_noisy}, this is almost exactly what we observe experimentally. Since LRGeomCG is more costly per iteration, it will have to converge much faster in order to compete with LMAFit. This is studied in detail in Section~\ref{sec:num_exp}.

\end{paragraph}

\section{Convergence for a modified cost function}\label{sec:convergence}

In this section, we show the convergence of Algorithm~\ref{alg:lrgeomcg} applied to a \emph{modified} cost function.

\subsection{Reconstruction on the tangent space}

As first step, we apply the general convergence theory for Riemannian optimization in \cite{AMS2008}. In particular, since $R_X$ is a smooth retraction, the Armijo-type line search together with the gradient-related search directions, allows us to conclude that any limit point of Algorithm~\ref{alg:lrgeomcg} is a critical point.
\begin{prpstn}\label{prop:accum_is_critical}
Let ${X_i}$ be an infinite sequence of iterates generated by Algorithm \ref{alg:lrgeomcg}. Then, every accumulation point ${X_*}$  of ${X_i}$ satisfies $\proj_{T_{X_*} \Manifold} \proj_{\Omega}(X_*) = \proj_{T_{X_*} \Manifold} \proj_{\Omega}(A)$.
\end{prpstn}
\begin{proof}
	From Theorem~4.3.1 in \cite{AMS2008}, every accumulation point is a critical point of the cost function $f$ of Algorithm~\ref{alg:lrgeomcg}. These critical points ${X_*}$ are determined by $\Grad f(X)=0$. By \eqref{eq:riem_gradient} this gives $\proj_{T_{X_*} \Manifold} \proj_{\Omega}(X_*-A) = 0$.
\qquad\end{proof}		

The next step is establishing that there exist limit points. However, since $\mathcal{M}_k$ is open---its closure are the matrices of rank \emph{bounded} by $k$---it is not guaranteed that any infinite sequence has indeed a limit point in $\mathcal{M}_k$.  To guarantee that there exists at least one such point in the sequence of iterates of Algorithm~\ref{alg:lrgeomcg}, we want to exploit that these iterates stay in a compact subset of $\mathcal{M}_k$.

It does not seem possible to show this for the objective function $f$ of Algorithm \ref{alg:lrgeomcg} without modifying the algorithm or making further assumptions. We therefore show that \emph{this property holds when Algorithm~\ref{alg:lrgeomcg} is applied not to $f$ but the regularized version}
\[
 g \colon \mathcal{M} \to \RSpace, \, X \mapsto f(X) +  \mu^2(\normF{X^\dagger}^2 + \normF{X}^2), \quad \mu>0.
\]
Since every $X \in \mathcal{M}_k$ is of rank $k$, the pseudo-inverse is smooth on $\mathcal{M}_k$ and hence $g(X)$ is also smooth. Using \cite[Thm.~4.3]{Golub:1973}, one can show that the gradient of $X \mapsto \normF{X^\dagger}^2 + \normF{X}^2$ equals $2U(\Sigma - \Sigma^{-3})V^T$ for $X=U\Sigma V^T$; hence, $\Grad g(X)$ can be computed very cheaply after modifying step 2 in Algorithm~\ref{alg:gradient}.

\begin{prpstn}\label{prop:conv_reg}
Let ${X_i}$ be an infinite sequence of iterates generated by Algorithm \ref{alg:lrgeomcg} but with the objective function $g(X) = f(X) + \mu^2(\normF{X^\dagger}^2 + \normF{X}^2)$ for some $0<\mu<1$. Then, $\lim_{i \to \infty} \proj_{T_{X_i} \Manifold} \proj_{\Omega}(X_i - A) =  2 \mu^2 (X^\dagger_i + X_i)$.
\end{prpstn}
\begin{proof}
  We will first show that the iterates stay in a closed and bounded subset of $\mathcal{M}_k$. Let ${L} = \{ X \in \mathcal{M}_k: g(X) \leq g(X_0) \}$ be the level set at $X_0$. By construction of the line search, all $X_i$ stay inside ${L}$ and we get
  \[
   \frac{1}{2}\normF{\proj_{\Omega}(X_i-A)}^2 + \mu^2 \normF{X^\dagger_i}^2 + \mu^2 \normF{X_i}^2 \leq C_0^2, \quad i > 0,
  \]
  where $C_0^2:= g(X_0)$. This implies
  \[
   \mu^2 \normF{X_i}^2 \leq C_0^2 - \frac{1}{2} \normF{\proj_{\Omega}(X_i-A)}^2 - \mu^2 \normF{X^\dagger_i}^2 \leq C_0^2
  \]
  and we obtain an upper bound on the largest singular value of each $X_i$:
  \[
   \sigma_1(X_i) \leq \normF{X_i} \leq C_0/\mu =: C^\sigma.
  \]
  Similarly, we get a lower bound on the smallest singular value:
  \[
    \mu^2 \normF{X^\dagger_i}^2 = \sum_{j=1}^k \frac{\mu^2 }{\sigma^2_j(X_i)} \leq C_0^2 -  \frac{1}{2} \normF{\proj_{\Omega}(X_i-A)}^2 - \mu^2 \normF{X_i}^2 \leq C_0^2,
  \]
  which implies that 
  \[
  \sigma_k(X_i) \geq \mu/C_0 =: C_\sigma.
  \]
  It is clear that all $X_i$ stay inside the set
  \[
   {B} = \{ X \in \mathcal{M}_k: \sigma_1(X) \leq C^\sigma ,\  \sigma_k(X) \geq C_\sigma \}.
  \]
  This set is closed and bounded, hence compact. 
  
  Now suppose that $\lim_{i\to\infty}\normF{\Grad g(X_i)} \neq 0$. Then there is a subsequence in $\{ X_i \}_{i\in\mathcal{K}}$ such $\normF{\Grad g(X)} \geq \varepsilon > 0$ for all $i\in\mathcal{K}$. Since $X_i \in  {B}$, this subsequence $\{ X_i \}_{i\in\mathcal{K}}$ has a limit point $X_*$ in ${B}$. By continuity of $\Grad g$, this implies that $\normF{\Grad g(X_*)} \geq \varepsilon$ which contradicts Theorem~4.3.1 in \cite{AMS2008} that every accumulation point is a critical point of $g$. Hence, $\lim_{i\to\infty}\normF{\Grad g(X_i)} = 0$.
\qquad\end{proof}			

In theory, $\mu^2$ can be very small, so one can choose it arbitrarily small, for example, as small as $\varepsilon_{\text{mach}} \simeq 10^{-16}$. This means that as long as $\sigma_1(X_i) = O(1)$ and $\sigma_k(X_i) \gg \varepsilon_{\text{mach}}$, the regularization terms in $g$ are negligible and one might as well have optimized the original objective function $f$, that is, $g$ with $\mu=0$. This is what we observed in all numerical experiments when $\rank(A) \geq k$. In case $\rank(A) < k$, it is obvious that now $\sigma_k(X_i) \to 0$ as $i \to \infty$. Theoretically one can still optimize $g$, although in practice it makes more sense to transfer the original optimization problem to the manifold of rank $k-1$ matrices. Since these rank drops do not occur in our experiments, proving convergence of such a method is beyond the scope of this paper. 

A similar argument appears in the analysis of other methods for rank constrained optimization too. For example, the proof of convergence in \cite{BurMon2005} for SDPLR \cite{Burer:2003} uses a regularization by $\mu \det(\Sigma)$ with $\Sigma$ containing the non-zero singular values of $X$; yet in practice, ones optimizes without this regularization using the observation that $\mu$ can be made arbitrarily small and in practice one always observes convergence. Analogous modifications also appear in \cite{Keshavan:2010a, Boumal:2011, Boumal:2012fk}.

\subsection{Reconstruction on the whole space}

Based on Proposition~\ref{prop:conv_reg}, we have that, for a suitable choice of $\mu$, any limit point satisfies 
\[
\normF{\proj_{T_{X_*} \Manifold} \proj_{\Omega}(X_* - A)} \leq \varepsilon, \quad \varepsilon > 0.
\]
In other words, we will have approximately fitted $X$ to the data $A$ on the image of $\proj_{T_{X_*} \Manifold} \proj_\Omega$ and, hopefully, this is sufficient to have $X_*=A$ on the whole space $\RSpace^{m \times n}$. Without further assumptions on $X_*$ and $A$, however, it is not reasonable to expect that $X_*$ equals $A$ since $\proj_{\Omega}$ usually has a very large null space. 

Let us split the error $E := X-A$ into the three quantities
\begin{align*}
	E_1 = \proj_{T_{X} \Manifold} \proj_\Omega (X-A), \ \ 
	E_2 = \proj_{N_{X} \Manifold} \proj_\Omega (X-A), \ \ \text{ and }\ \ 
	E_3 = \proj^\perp_\Omega (X-A),
\end{align*}
where $N_{X} \Manifold$ is the normal space of $X$, i.e., the subspace of all matrices $Z \in \RSpace^{m \times n}$ orthogonal to $T_{X} \Manifold$. 

Upon convergence of Algorithm~\ref{alg:lrgeomcg}, $\normF{E} \to \varepsilon$, but how big can $E_2$ be? In general, $E_2$ can be arbitrary large since it is the Lagrange multiplier of the fixed rank constraint of $X$. In fact, when the rank of $X$ is smaller than the exact rank of $A$, $E_2$ typically never vanishes, even though $E_1$ can converge to zero. On the other hand, when the ranks of $X$ and $A$ are the same, one typically observes that $\normF{E} \to 0$ implies $\normF{E} \to 0$. This can be easily checked during the course of the algorithm, since $\proj_\Omega (X-A) = E_1 + E_2$ is computed explicitly for the computation of the gradient in Algorithm~\ref{alg:gradient}. 

Suppose now that $\normF{E_1} \leq \varepsilon$ and $\normF{E_2} \leq \tau$, which means that $X$ is in good agreement with $A$ on $\Omega$. The hope is that $X$ and $A$ agree on the complement of $\Omega$ too. This is of course not always the case, but it is exactly the crucial assumption of low-rank matrix completion that the observations on $\Omega$ are sufficient to complete $A$. To quantify this assumption, one can make use of standard theory in matrix completion literature; see, e.g., \cite{Candes:2009, Candes:2009b, Keshavan:2010}. Since it is not the purpose of this paper to analyze the convergence properties of Algorithm~\ref{alg:lrgeomcg} for various, rather theoretical random models, and we do not rely on this theory in the numerical experiments later on, we omit this discussion.

\section{Numerical experiments} \label{sec:num_exp}

The implementation of LRGeomCG, i.e., Algorithm~\ref{alg:lrgeomcg}, was done in {\sc Matlab} R2012a on a desktop computer
with a 3.10 GHz CPU and 8 GB of memory. All reported times are wall-clock  time that include the setup phase of the solvers (if needed) but exclude setting up the (random) problem. 

As observed in \cite{Boumal:2011}, the performance of most matrix completion solvers behaves very differently for highly non-square matrices; hence, as is mostly done in the literature, we focus only on square matrices. In our numerical experiments below we only compare with LMAFit of \cite{Wen:2010}. Extensive comparison with other approaches based in \cite{Cai:2010,Ma:2011,Meka:2010,Goldfarb:2011,Lin:2009a,Toh:2010,Liu:2010} revealed that LMAFit is overall the fastest method for square matrices. We refer to \cite{Michenkova:2011} for these results. 

The implementation of LMAFit was the one provided by the authors\footnote{Version of August 14, 2012 downloaded from \url{http://lmafit.blogs.rice.edu/}.}. All options were kept the same, except rank adaptivity was turned off. In Sections~\ref{sec:exp_noise} and~\ref{sec:exp_decay} we also changed the stopping criteria as explained there. In order to have a fair comparison between LMAFit and LRGeomCG, the operation $\proj_\Omega$ was performed by the same {\sc Matlab} function \verb|part_XY.m| in both solvers. In addition, the initial guesses were taken as random rank $k$ matrices (more on this below).

In the first set of experiments, we will complete random low-rank matrices which are generated as proposed in \cite{Cai:2010}. First, construct two random matrices $A_L, A_R \in \RSpace^{n \times k}$ with i.i.d.\ standard Gaussian entries; then, assemble $A \defeq A_L A_R^T$; finally, the set of observations $\Omega$ is sampled uniformly at random among all sets of cardinality $|\Omega|$. The resulting observed matrix to complete is now $A_{\Omega} \defeq \proj_{\Omega}(A)$. Standard random matrix theory asserts that $\normF{A} \simeq n \sqrt{k}$ and $\normF{A_{\Omega}} \simeq \sqrt{|\Omega|k}$. In the later experiments, the test matrices are different and their construction will be explained there.

When we report on the relative error of an approximation $X$, we mean the error
\begin{equation}\label{eq:rel_residual}
 \text{relative error} = \normF{X - A} / \normF{A} 
\end{equation}
computed on all the entries of $A$. On the other hand, the algorithms will compute a relative residual on $\Omega$ only:
\begin{equation}\label{eq:rel_error}
 \text{relative residual} = \normF{\proj_{\Omega}(X - A) } / \normF{ \proj_{\Omega}(A)}.
\end{equation}
Unless stated otherwise, we set as tolerance for the solvers a relative residual of $10^{-12}$. Such a high precision is necessary to study the asymptotic convergence rate of the solvers, but where appropriate we will draw conclusions for moderate precisions too. 

As initial guess $X_1$ for Algorithm~\ref{alg:lrgeomcg}, we construct a random rank-$k$ matrix following the same procedure as above for $M$. The oversampling factor (OS) for a rank-$k$ matrix is defined as the ratio of the number of samples to the degrees of freedom in a non-symmetric matrix of rank $k$,
\begin{equation}\label{eq:OS}
 \text{OS} = | \Omega | / (k(2n-k)).
\end{equation}
Obviously one needs at least $\text{OS} \geq 1$ to uniquely complete any rank-$k$ matrix. In the experiments below, we observe that, regardless of the method, $\text{OS} > 2$ is needed to reliably recover an incoherent matrix of rank $k$ after uniform sampling. In addition, each row and each column needs to be sampled at least once. By the coupon's collector problem, this requires that $| \Omega | > C n \log(n)$ but in all experiments below, $\Omega$ is always sufficiently large such that each row and column is sampled at least once. Hence, we will only focus on the quantity $\text{OS}$.

\subsection{Influence of size and rank for fixed oversampling}

As first test, we complete random matrices $A$ of exact rank $k$ with LRGeomCG and LMAFit and we explicitly exploit the knowledge of the rank in the algorithms. The purpose of this test is to investigate the dependence of the algorithms on the size and the rank of the matrices. In Section~\ref{sec:exp_oversampling}, we will study the influence of $\text{OS}$, but for now we fix the oversampling to $\text{OS}=3$. In Table~\ref{tab:conv_rank_known}, we report on the mean run time and the number of iterations taking over $10$ random instances.  The run time and the iteration count of all these instances are visualized in Figures~\ref{fig:its_rank_known} and \ref{fig:times_rank_known}, respectively.

Overall, LRGeomCG needs less iterations than LMAFit for all problems. In Figure~\ref{fig:its_rank_known}, we see that this is because LRGeomCG converges faster asymptotically although the convergence of LRGeomCG exhibits a slower transient behavior in the beginning. This phase is, however, only limited to the first 20 iterations. Since the cost per iteration for LMAFit is cheaper than for LRGeomCG (see the estimates~\eqref{eq:total_flops_practice} and~\eqref{eq:total_flops_lma}), this transient phase leads to a trade-off between runtime and precision. This is clearly visible in the timings of Figure~\ref{fig:times_rank_known}: When sufficiently high accuracy is requested, LRGeomCG is always faster, whereas for low precision, LMAFit is faster.

Let us now check in more detail the dependence on the size $n$. It is clear from the left sides of Table~\ref{tab:conv_rank_known} and Figure~\ref{fig:its_rank_known} that the iteration counts of LMAFit and LRGeomCG grow with $n$ but seem to stagnate as $n \to \infty$. In addition, LMAFit needs about three times more iterations than LRGeomCG. Hence, even though each iteration of LRGeomCG is about two times more expensive than one iteration of LMAFit, LRGeomCG will eventually always be faster than LMAFit for sufficiently high precision. In Figure~\ref{fig:times_rank_known} on the left, one can, for example, see that this trade-off point happens at a precision of about $10^{-5}$. For growing rank $k$, the conclusion of the same analysis is very much the same.
 
The previous conclusion depends critically on the convergence factor and, as we will see later on, this factor is determined by the amount of oversampling $\text{OS}$. But for difficult problems ($\text{OS}=3$ is near the lower limit of $2$), we can already observe that LRGeomCG is about 50\% faster than LMAFit.

\begin{figure}
  \begin{minipage}[b]{0.49\linewidth}
    \centering
    \includegraphics[width=\linewidth]{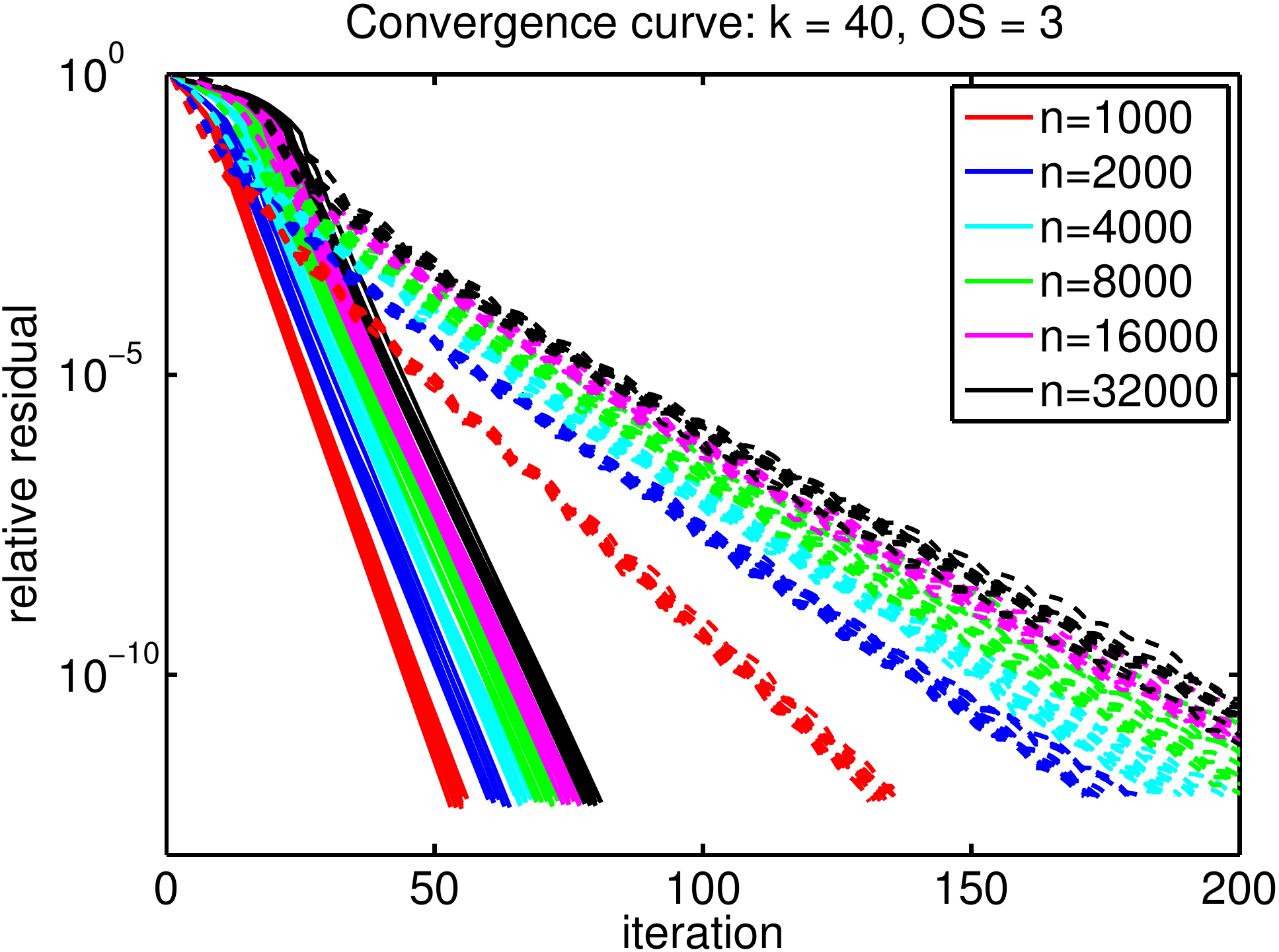}
  \end{minipage}
  \hfill
  \begin{minipage}[b]{0.49\linewidth}
    \centering
    \includegraphics[width=\linewidth]{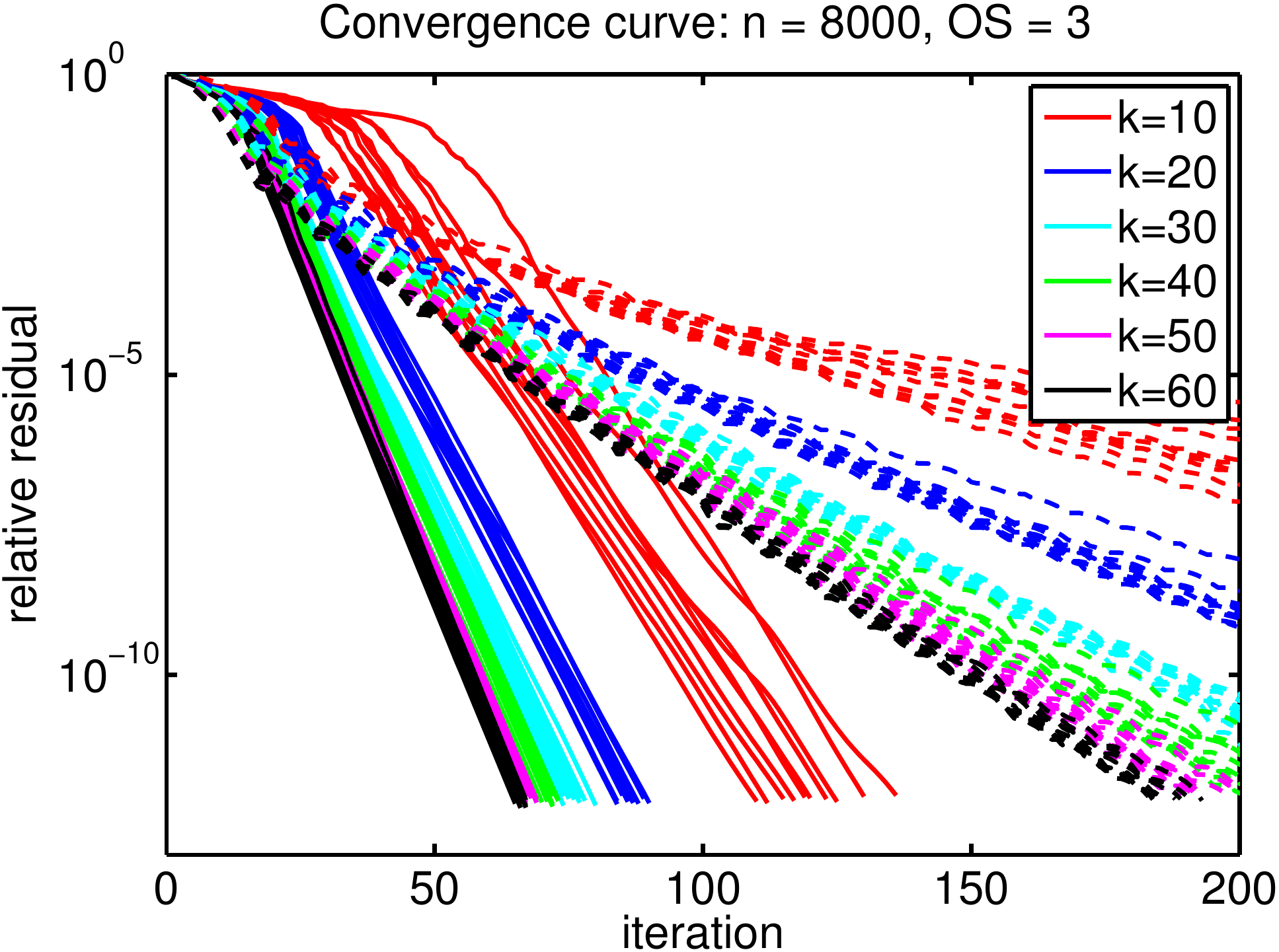}
  \end{minipage}  
  \caption{Convergence curves for LMAFit (dashed line) and LRGeomCG (full line) for fixed oversampling of $\text{OS}=3$. Left: variable size $n$ and fixed rank $k=40$; right: variable ranks $k$ and fixed size $n=8000$.}\label{fig:its_rank_known}
\end{figure}

\begin{figure}
  \begin{minipage}[b]{0.49\linewidth}
    \centering
    \includegraphics[width=\linewidth]{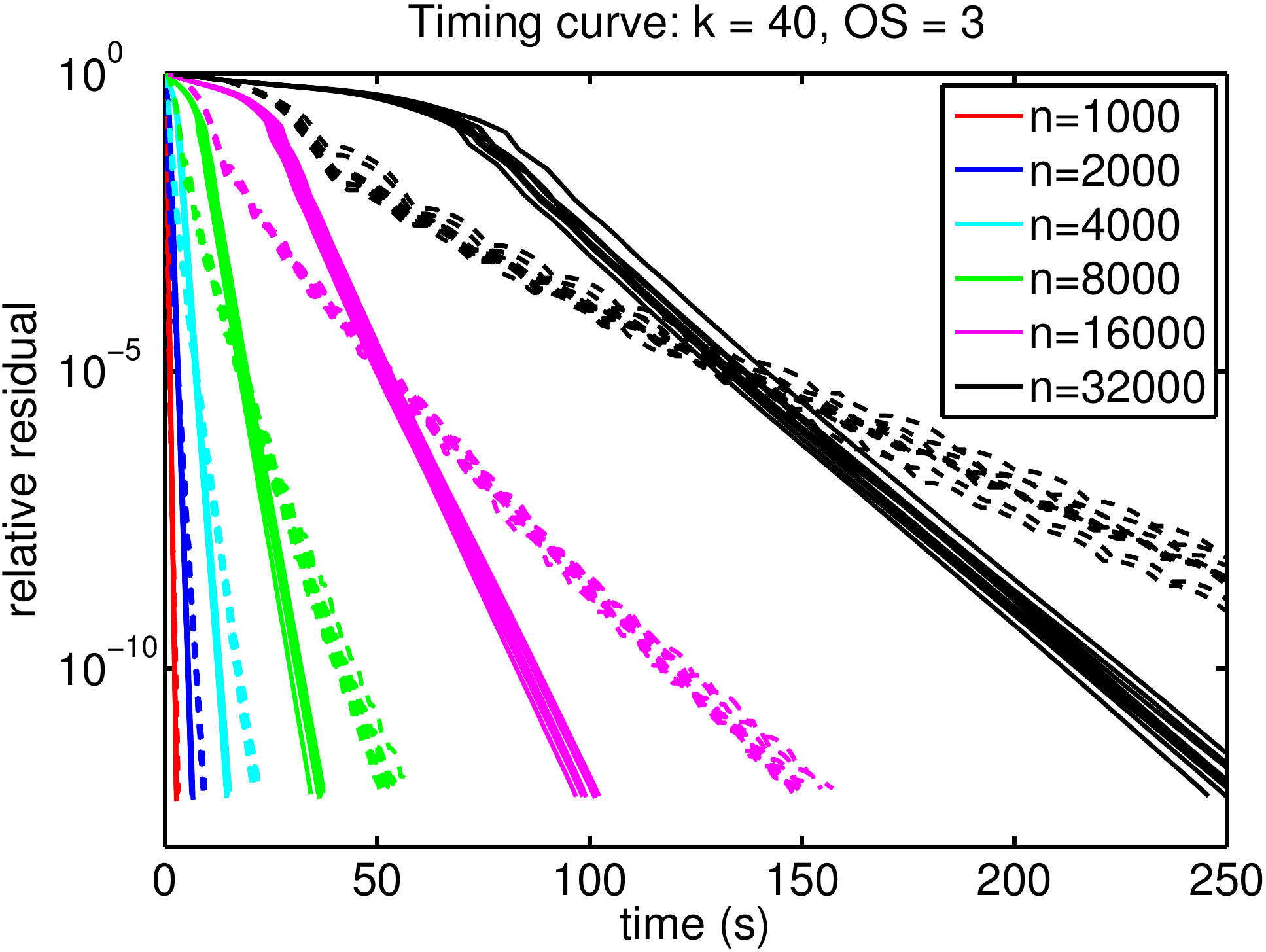}
  \end{minipage}
  \hfill
  \begin{minipage}[b]{0.49\linewidth}
    \centering
    \includegraphics[width=\linewidth]{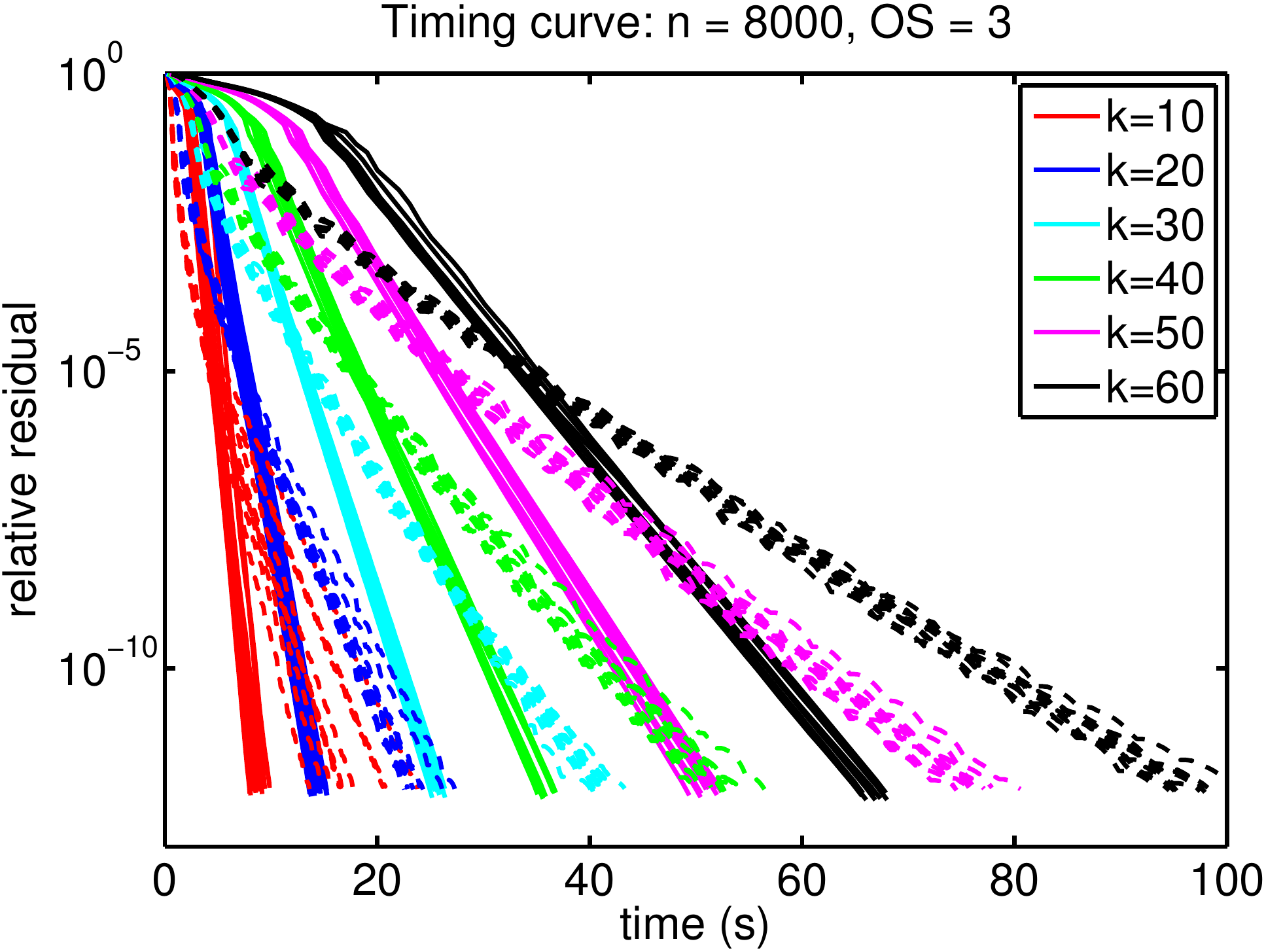}
  \end{minipage}  
  \caption{Timing curves for LMAFit (dashed line) and LRGeomCG (full line) for fixed oversampling of $\text{OS}=3$. Left: variable size $n$ and fixed rank $k=40$; right: variable ranks $k$ and fixed size $n=8000$.}\label{fig:times_rank_known}
\end{figure}

\begin{table}
  \renewcommand{\arraystretch}{1.2}
	\footnotesize
	\caption{The mean of the computational results for Figure~\ref{fig:its_rank_known}--\ref{fig:times_rank_known} for solving with a tolerance of $10^{-12}$.}\label{tab:conv_rank_known}
  \centering%
\begin{tabular}{@{}c cc cc  l c cc cc@{}}
 \toprule
  \multicolumn{5}{c}{Fixed rank $k=40$} & \phantom{ab} & \multicolumn{5}{c}{Fixed size $n=8000$} \\ \cmidrule(r){1-5} \cmidrule(l){7-11}
  \multirow{2}{*}{$n$}       & \multicolumn{2}{c}{LMAFit} & \multicolumn{2}{c}{LRGeomCG} &   & \multicolumn{2}{c}{LMAFit} & \multicolumn{2}{c}{LRGeomCG} \\
         & time(s.)   & \#its. & time(s.) & \#its. &  & \phantom{00}$k$\phantom{00}  & time(s.)   &  \#its. & time(s.) & \#its.     \\ \midrule
  1000   & 3.11    & 133  & {\bf 2.74}  & 54.5   &&  10    &  17.6  &  518  &  {\bf 8.81}   &  121\\
  2000   & 9.13    & 175  & {\bf 6.43}  & 61.3   &&  20    &  23.9  &  297  &  {\bf 14.6}   &  86.5 \\
  4000   & 21.2    & 191  & {\bf 14.8}  & 66.7   &&  30    &  40.9  &  232  &  {\bf 25.8}   &  76.1 \\
  8000   & 53.3    & 211  & {\bf 36.5}  & 71.7   &&  40    &  52.7  &  211  &  {\bf 35.8}   &  71.7 \\
  16000  & 150     & 222  & {\bf 99.1}  & 75.4   &&  50    &  76.1  &  194  &  {\bf 51.2}   &  67.7 \\
  32000  & 383     & 233  & {\bf 254}   & 79.1   &&  60    &  96.6  &  186  &  {\bf 67.0}   &  66.1 \\  \bottomrule
\end{tabular}  
\end{table}

\subsection{Hybrid strategy}

The experimental results from above clearly show that the transient behavior of LRGeomCG's convergence is detrimental if only modest accuracy is required. The reason for this slow phase is the lack of non-linear CG acceleration ($\beta$ in Algorithm~\ref{alg:beta_adjust} is almost always zero) which essentially reduces LRGeomCG to a steepest descent algorithm. On the other hand, LMAFit is much less affected by slow convergence during the initial phase.

A simple heuristic to overcome this bad phase is a hybrid solver: For the first $I$ iterations, we use LMAFit; after that, we hope that the non-linear CG acceleration kicks in immediately and we can efficiently solve with LRGeomCG. In Figure~\ref{fig:times_hybrid}, we have tested this strategy  for different values of $I$ on a problem of the previous section with $n=16000$ and $k=40$.

From the left panel of Figure~\ref{fig:times_hybrid}, we get that the run time to solve for a tolerance of $10^{-12}$ is reduced from 97 sec.~(LRGeomCG) and 147 sec.~(LMAFit) to about 80 sec.~for the choices $I=10,20,30,40$. Performing only one iteration of LMAFit is less effective. 
For $I=20$, the hybrid strategy has almost no transient behavior anymore and it is always faster than LRGeomCG or LMAFit alone. Furthermore, the choice of $I$ is not very sensitive to the performance as long as it is between 10 and 40, and already for $I=10$, there is a significant speedup of LRGeomCG noticeable for all tolerances.

The quantity $\beta$ of PR+ in Algorithm~\ref{alg:beta_adjust} is plotted in the right panel of Figure~\ref{fig:times_hybrid}. Until the 20th iteration, plain LRGeomCG shows no meaningful CG acceleration. For the hybrid strategies with $I>10$, the acceleration kicks in almost immediately. Observe that all approaches converge to $\beta\simeq 0.4$.

While this experiment shows that there is potential to speed up LRGeomCG in the early phase, we do not wish to claim that the present strategy is very robust nor directly applicable. The main point that we wish to make, however, is that the slow phase can be avoided in a relatively straightforward way, and that LRGeomCG can be warm started by any other method. In particular, this shows that LRGeomCG can be efficiently employed once the correct rank of the solution $X$ is identified by some other method. As explained in the introduction, there are numerous other methods for low-rank matrix completion that can reliably identify this rank but have a slow asymptotic convergence. Combining them with LRGeomCG seems like a promising way forward. 

\begin{figure}
  \begin{minipage}[b]{0.49\linewidth}
    \centering
    \includegraphics[width=\linewidth]{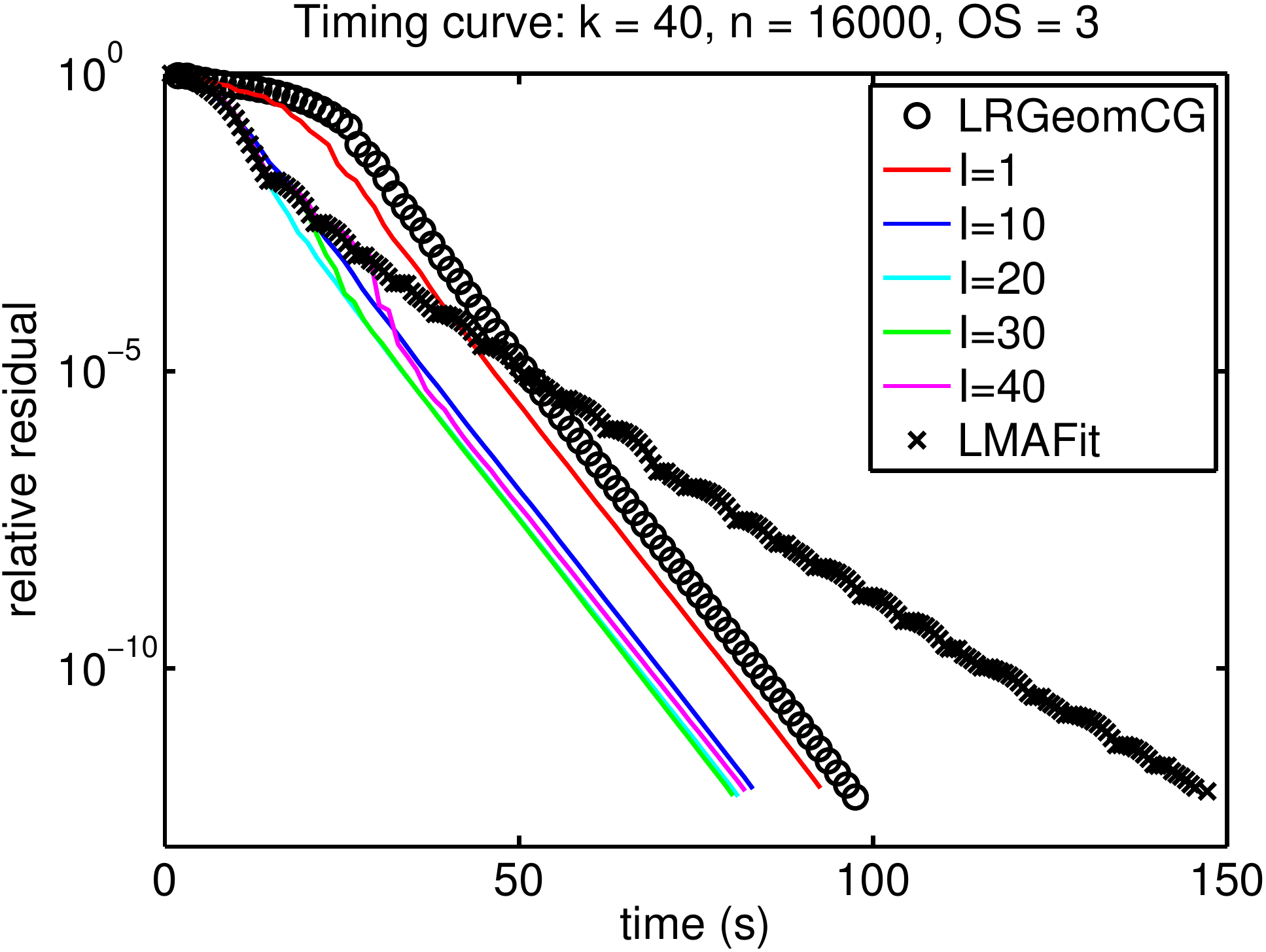}
  \end{minipage}
  \hfill
  \begin{minipage}[b]{0.49\linewidth}
    \centering
    \includegraphics[width=\linewidth]{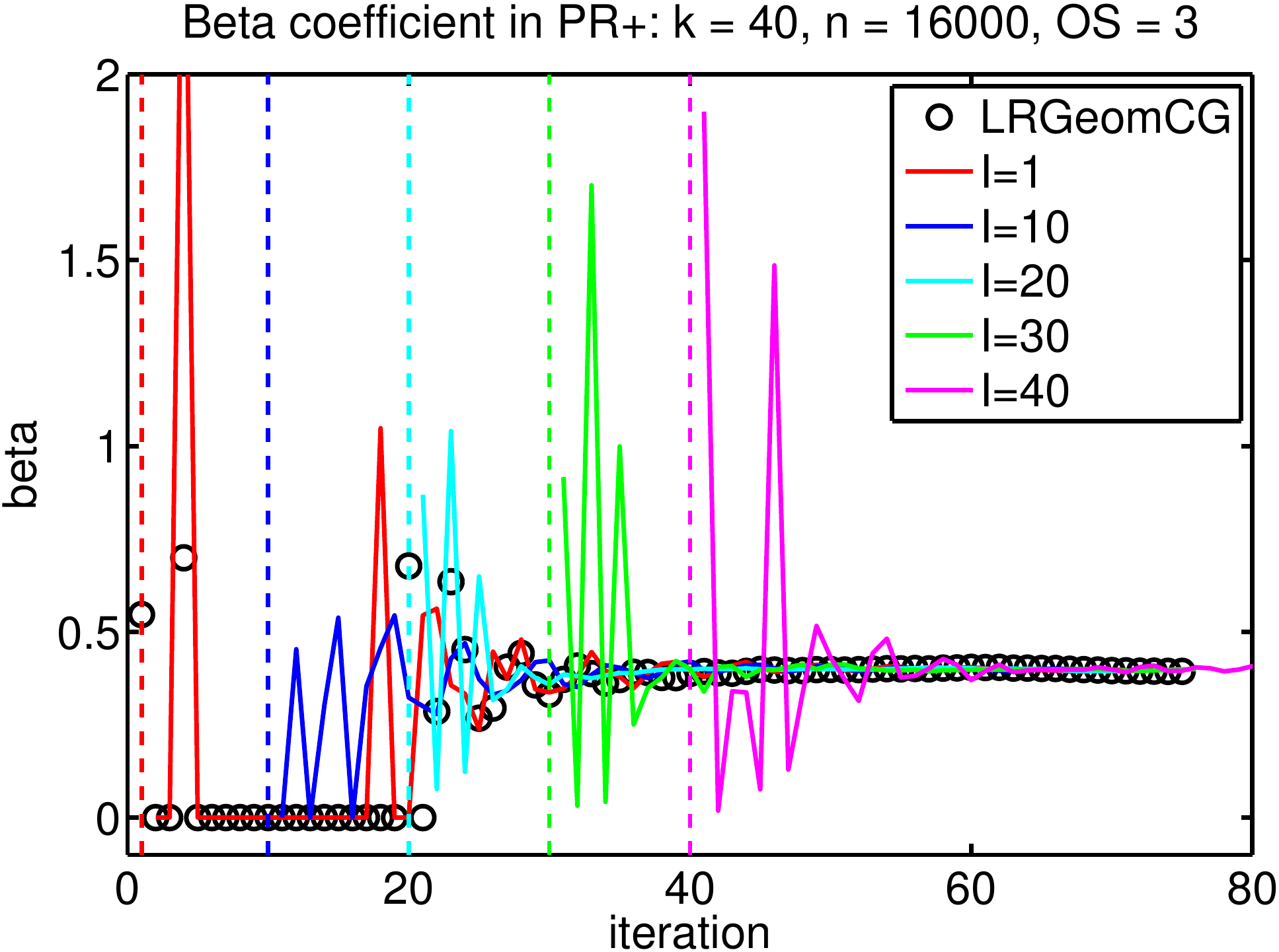}
  \end{minipage}  
  \caption{Timing curve and $\beta$ coefficient of the hybrid strategy for different values of $I$.}\label{fig:times_hybrid}
\end{figure}

\subsection{Influence of oversampling}\label{sec:exp_oversampling}

Next, we investigate the influence of oversampling on the convergence speed. We took 10 random problems with fixed rank and size, but vary the oversampling factor $\text{OS}$ using 50 values between $1, \ldots, 15$. In Figure~\ref{fig:times_fr} on the left, the asymptotic convergence factor $\rho$ is visible for LMAFit and LRGeomCG for three different combinations of the size and rank. Since the convergence is linear and we want to filter out any transient behavior, this factor was computed as
\[
\rho = \left( \frac{ \normF{\proj_{\Omega}(X_{i_{\text{end}}} - A)} }{ \normF{\proj_{\Omega}(X_{10} - A)} } \right)^{1/( i_{\text{end}} - 10 )}
\]
where $i_{\text{end}}$ indicates the last iteration. A factor of one indicates failure to converge within $4000$ steps. One can observe that all methods become slower as $\text{OS} \to 2$. Further, the convergence factor of LMAFit seems to stagnate for large values of $\text{OS}$ while LRGeomCG's actually becomes better. Since for growing $\text{OS}$ the completion problems become easier as more entries are available, only LRGeomCG shows the expected behavior. In contrast to our parameter-free choice of non-linear CG, LMAFit needs to determine and adjust a certain acceleration factor dynamically based on the performance of the iteration. We believe the stagnation in Figure~\ref{fig:times_fr} on the left is due to a suboptimal choice of this factor in LMAFit. 

In the right panel of Figure~\ref{fig:times_fr}, the mean time to decrease the relative residual by a factor of $10$ is displayed. These timings were again determined by neglecting the first 10 iterations and then interpolating the time needed for a reduction of the residual by $10$. Similarly as before, we can observe that since the cost per iteration is cheaper for LMAFit, there is a range for $\text{OS}$ around $7 \ldots 9$ where LMAFit is only slightly slower than LRGeomCG. However, for smaller and larger values of $\text{OS}$, LRGeomCG is always faster by a significant margin (observe the logarithmic scale).

 \begin{figure}
   \begin{minipage}[b]{0.49\linewidth}
     \centering
     \includegraphics[width=\linewidth]{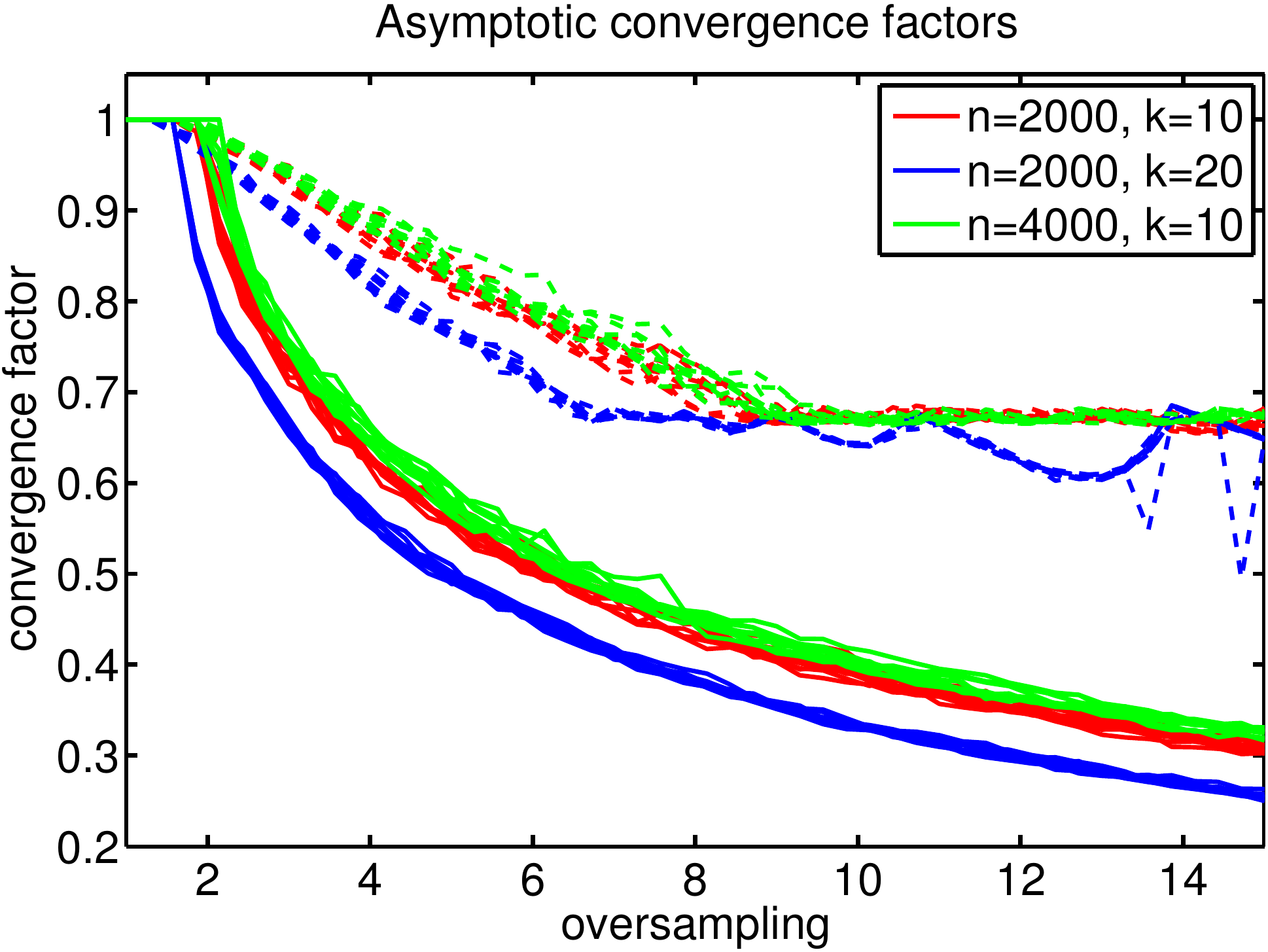}
   \end{minipage}
   \hfill
   \begin{minipage}[b]{0.49\linewidth}
     \centering
     \includegraphics[width=\linewidth]{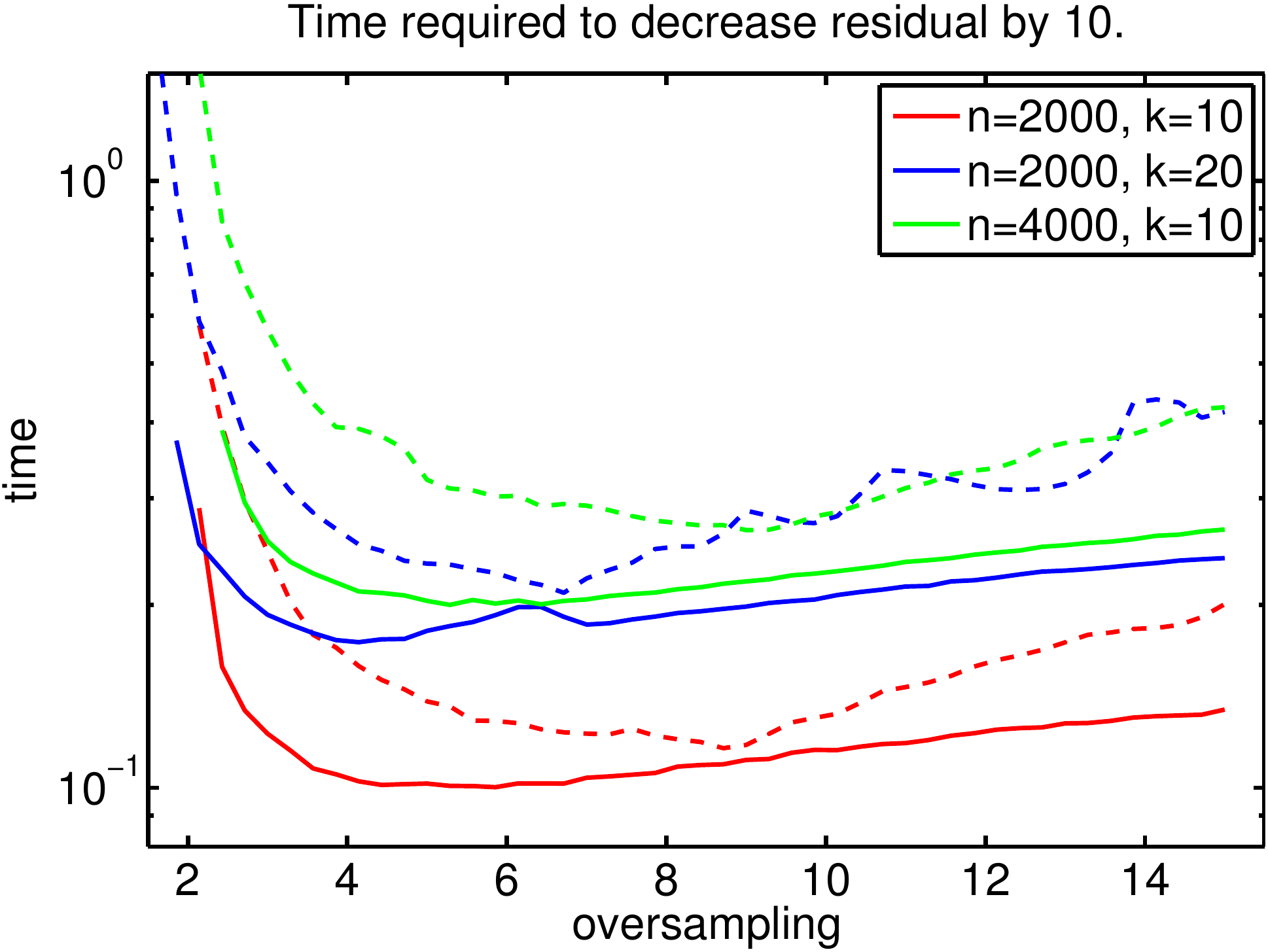}
   \end{minipage}  
   \caption{Asymptotic convergence factor $\rho$ (left) and mean time to decrease the error by a factor of 10 (right) for LMAFit (dashed line) and LRGeomCG (full line) in function of the oversampling rate $\text{OS}$.}\label{fig:times_fr}
 \end{figure}

\subsection{Influence of noise}\label{sec:exp_noise}

Next, we investigate the influence of noise by adding random perturbations to the rank-$k$ matrix $A$. We define the {noisy matrix} $A^{(\varepsilon)}$ with {noise level} $\varepsilon$ as 
\[
 A^{(\varepsilon)} \defeq A + \varepsilon \, \frac{ \normF{A_{\Omega}}}{\normF{N_{\Omega}}} \, N,
\]
where $N$ is a standard Gaussian matrix and $\Omega$ is the usual random sampling set (see also \cite{Toh:2010,Wen:2010} for a similar setup). The reason for defining $A^{(\varepsilon)}$ in this way is that we have
\[
 \normF{\proj_\Omega(A - A^{(\varepsilon)})} = \varepsilon \, \frac{ \normF{A_{\Omega}}}{\normF{N_{\Omega}}} \normF{\proj_\Omega(N)} \simeq \varepsilon\,\sqrt{|\Omega|k}.
\]
So, the best relative residual we may expect from an approximation $X^{\text{opt}} \simeq A$ is
\begin{equation}\label{eq:residual_noisy}
 \normF{\proj_\Omega(X^{\text{opt}} - A^{(\varepsilon)})} / \normF{\proj_\Omega(A)} \simeq \normF{\proj_\Omega(A - A^{(\varepsilon)})} / \normF{\proj_\Omega(A)} \simeq \varepsilon.
\end{equation}
Similarly, the best relative error of $X^{\text{opt}}$ should be on the order of the noise ratio too since
\begin{equation}\label{eq:error_noisy}
 \normF{A - A^{(\varepsilon)}} / \normF{A} \simeq \varepsilon.
\end{equation}

Due to the presence of noise, the relative residual \eqref{eq:rel_residual} cannot go to zero but the Riemannian gradient of our optimization problem can. In principle, this suffices to detect convergence of LRGeomCG. In practice, however, we have noticed that this wastes a lot of iterations in case the iteration stagnates. A simply remedy is to monitor
\begin{equation}\label{eq:stagnation}
  \text{relative change at step $i$} = \left\vert 1 - \sqrt{f(X_i)/f(X_{i-1})} \right\vert
\end{equation}  
and stop the iteration when this value drops below a certain threshold. After some experimentation, we fixed this threshold to $10^{-3}$. Such a stagnation detection is applied by most other low-rank completion solvers \cite{Toh:2010,Wen:2010}, although its specific form varies. Our choice coincides with the one from \cite{Wen:2010}, but we have lowered the threshold.

After equipping LMAFit and LRGeomCG with this stagnation detection, we can display the experimental results on Figure~\ref{fig:its_rank_noise} and Table~\ref{tab:conv_rank_known_noisy}. We have only reported on one choice for the rank, size and oversampling since the same conclusions can be reached for any other choice. Based on Figure~\ref{fig:its_rank_noise}, it is clear that the stagnation is effectively detected by both methods, except for the very noisy case of $\varepsilon=1$. (Recall that we changed the original stagnation detection procedure in LMAFit to ours of \eqref{eq:stagnation}, to be able to draw a fair comparison.)

Comparing the results with those of the previous section, the only difference is that the iteration stagnates when the error reaches the noise level. In particular, the iterations are undisturbed by the presence of the noise up until the very last iterations. In Table~\ref{tab:conv_rank_known_noisy}, one can observe that both methods solve all problems up to the noise level, which is the best that can be expected from \eqref{eq:residual_noisy} and \eqref{eq:error_noisy}. Again, LRGeomCG is faster when a higher accuracy is required, which, in this case, corresponds to small noise levels. Since the iteration is unaffected by the noise until the very end, the hybrid strategy of the previous section should be effective here too.

    \begin{figure}
      \centering
    \includegraphics[width=0.55\linewidth]{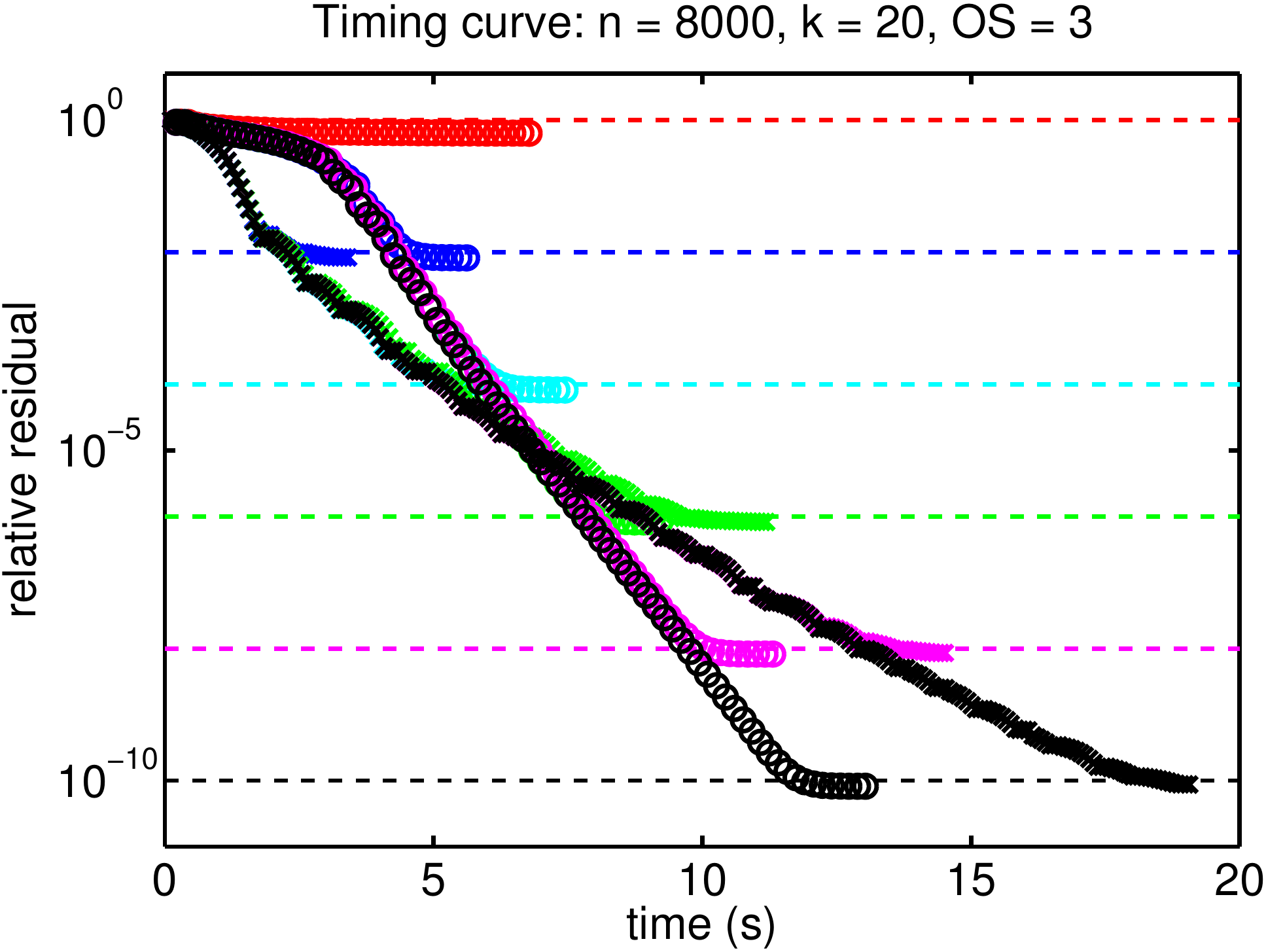}
    \caption{Timing curves for LMAFit (crosses) and  LRGeomCG (circles) for different noise levels with size $n=8000$, rank $k=20$, fixed oversampling of $\text{OS}=3$. The noise levels $\varepsilon$ are indicated by dashed lines in different color.}\label{fig:its_rank_noise}
  \end{figure}

  \begin{table}
    \renewcommand{\arraystretch}{1.2}
	\footnotesize
	  \caption{Computational results for Figure~\ref{fig:its_rank_noise}.}\label{tab:conv_rank_known_noisy}
    \centering%
  \begin{tabular}{@{}c l cccc l  cccc@{}}
    \toprule
    \multirow{2}{*}{$\varepsilon$}   & \phantom{a}    & \multicolumn{4}{c}{LMAFit} & \phantom{a} & \multicolumn{4}{c}{LRGeomCG}  \\ \cmidrule(r){3-6} \cmidrule(l){8-11}
      & & time(s.)  &  \#its. & error & residual  & & time(s.)  & \#its. & error & residual     \\\hline
    $10^{-0\phantom{0}}$  & & {\bf 2.6}    & 31   & $1.04\, \varepsilon$   &  $0.59\, \varepsilon$   & &  6.8    &  42  &  $1.52\, \varepsilon$ &  $0.63\, \varepsilon$  \\
    $10^{-2\phantom{0}}$  & & {\bf 3.4}   & 42   & $0.77\, \varepsilon$   &  $0.83\, \varepsilon$   & &  5.6    &  36  &  $0.72\, \varepsilon$ &  $0.82\, \varepsilon$  \\
    $10^{-4\phantom{0}}$  & & {\bf 6.4}    & 79  & $0.83\, \varepsilon$   &  $0.83\, \varepsilon$   & &  7.5    &  46  &  $0.72\, \varepsilon$ &  $0.82\, \varepsilon$  \\
    $10^{-6\phantom{0}}$  & & 11    & 133  & $0.74\, \varepsilon$   &  $0.82\, \varepsilon$   & &  {\bf 9.2}    &  58  &  $0.72\, \varepsilon$ &  $0.82\, \varepsilon$  \\
    $10^{-8\phantom{0}}$  & & 15    & 179  & $0.96\, \varepsilon$   &  $0.85\, \varepsilon$   & &  {\bf 11}    &  70  &  $0.72\, \varepsilon$ &  $0.82\, \varepsilon$  \\
    $10^{-10}$            & & 19    & 235  & $1.09\, \varepsilon$   &  $0.87\, \varepsilon$   & &  {\bf 13}    &  81  &  $0.72\, \varepsilon$ &  $0.82\, \varepsilon$  \\ \bottomrule
  \end{tabular}  
  \end{table}

\subsection{Exponentially decaying singular values}\label{sec:exp_decay}

In the previous problems, the rank of the matrices could be unambiguously defined. Even in the noisy case, there was still a sufficiently large gap between the original nonzero singular values and the ones affected by noise. In this section, we will complete a matrix for which the singular values decay but do not become exactly zero. 

Although a different setting than the previous problems, this type of matrices occurs frequently in the context of approximation of discretized functions or solutions of PDEs on tensorized domains; see, e.g., \cite{Kressner:2010c}. In this case, the problem of approximating discretized functions by low-rank matrices is primarily motivated by reducing the amount of data to store. On a continuos level, low-rank approximation corresponds in this setting to approximation by sums of separable functions. As such, one is typically interested in approximating the matrix up to a certain tolerance using the lowest rank possible.

The (exact) rank of a matrix is now better replaced by the numerical rank, or $\varepsilon$-rank \cite[Chapter 2.5.5]{Golub:1996}. Depending on the required accuracy $\varepsilon$, the $\varepsilon$-rank is the quantity
\begin{equation}\label{eq:eps_rank}
  k_{\varepsilon}(A) = \min_{\norm{A-B}_2 \leq \varepsilon} \rank(B).
\end{equation}
Obviously, when $\varepsilon=0$, one recovers the (exact) rank of a matrix. It is well known that the $\varepsilon$-rank of $A \in \RSpace^{m \times n}$ can be determined from the singular values $\sigma_i$ of $A$ as follows
\[
 \sigma_1 \geq \sigma_{k_{\varepsilon}} > \varepsilon \geq \sigma_{k_{\varepsilon}+1} \geq  \cdots \geq \sigma_p, \quad p = \min(m,n). 
\]

In the context of the approximation of bivariate functions, the relation \eqref{eq:eps_rank} is very useful in the following way. Let $f(x,y)$ be a function defined on a tensorized domain $(x,y) \in \mathcal{X} \times \mathcal{Y}$. After discretization of $x$ and $y$, we arrive at a matrix $A$. If we allow for an absolute error in the spectral norm of the size $\varepsilon$, then \eqref{eq:eps_rank} tells us that we can approximate $M$ by a rank $k_{\varepsilon}$ matrix. 

As example, we take the following simple bivariate function to complete:
\begin{equation}\label{eq:f_x_y}
 f(x,y) = \frac{1}{\sigma + \norm{x-y}^2_2}, 
\end{equation}
for some $\sigma > 0$. The parameter $\sigma$ controls the decay of the singular values. Such functions occur as two-point correlations related to second-order elliptic problems with stochastic source terms; see \cite{Harbrecht:2010}.

We ran LRGeomCG and LMAFit after discretizing \eqref{eq:f_x_y} with a matrix of size $n=8000$ using an oversampling factor \text{OS} of 8 (calculated for rank 20) and a decay of $\sigma=1$. Since a stopping condition on the residual has no meaning for this example, we only used a tolerance of $10^{-3}$ on the relative change \eqref{eq:stagnation} or a maximum of 500 iterations. In Figure~\ref{fig:rank_decay}, the final accuracy for different choices of the rank is visible for these two problem sets. 

The results labeled ``(no hom)'' use the standard strategy of starting with random initial guesses for each value of the rank $k$. Clearly, this results in a very unsatisfactory performance of LRGeomCG and LMAFit since the error of the completion does not decrease with increasing rank. Remark that we have measured the relative \emph{error} as
\[
 \| \proj_\Gamma(X_*-A) \|_\textrm{F} /  \| \proj_\Gamma(A) \|_\textrm{F}
\]
where $\Gamma$ is a random sampling set different from $\Omega$ but equally large.

  \begin{figure}[h!]
    \centering
  \includegraphics[width=0.55\linewidth]{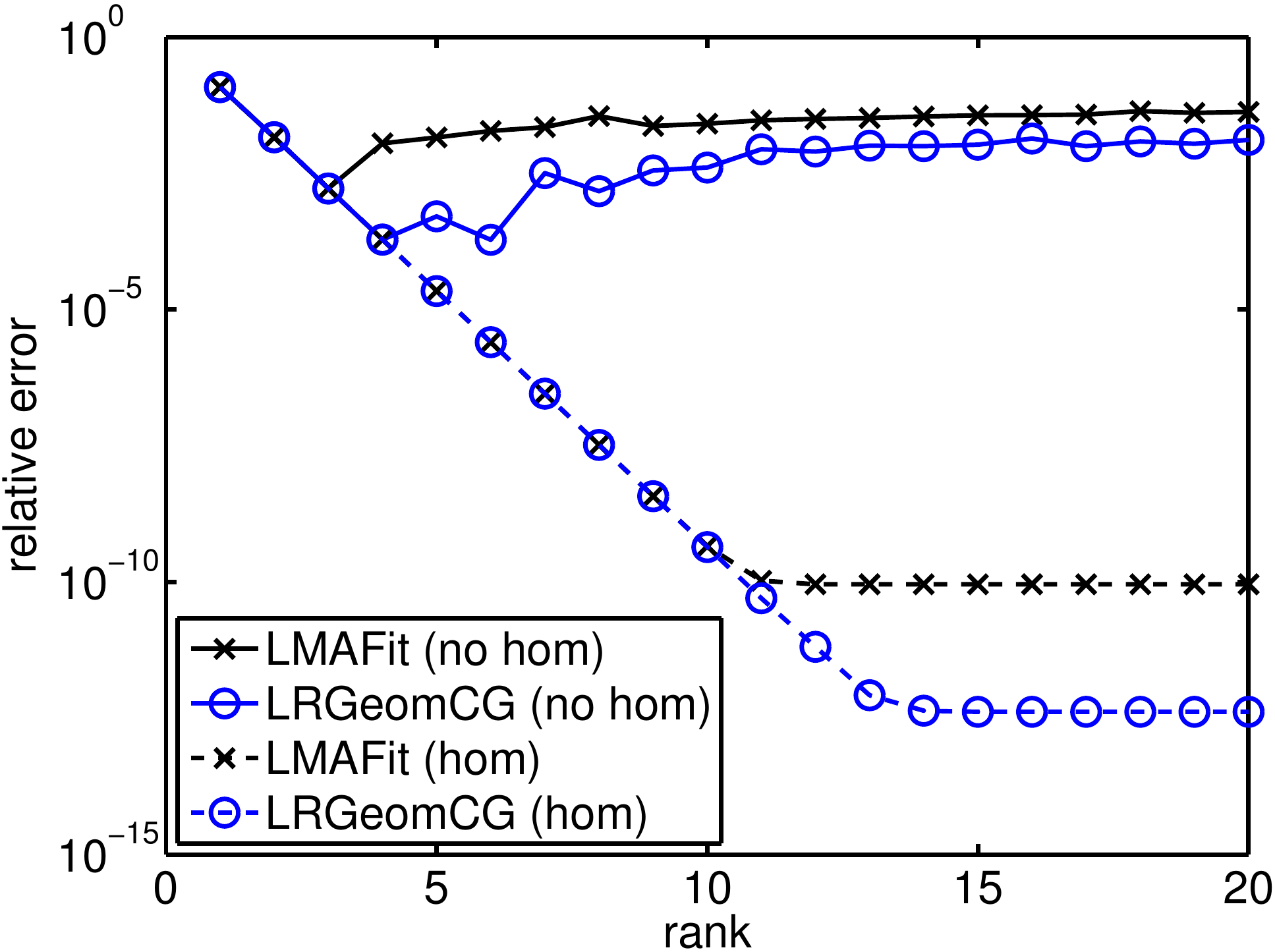}
  \caption{Relative error of LMAFit and  LRGeomCG after completion of a discretization of \eqref{eq:f_x_y} for OS=8 with and without a homotopy strategy.}\label{fig:rank_decay}
\end{figure}

The previous behavior is a clear example that the local optimizers of LMAFit and LRGeomCG are very far away from the global ones. However, there is a straightforward solution to this problem: Instead of taking a random initial guess for each $k$, we only take a random initial guess for $k=1$. For all other $k>1$, we use $X=U \Sigma V^T$ where
\[
 U = \begin{bmatrix} \widetilde U & u \end{bmatrix}, \quad  \Sigma = \begin{bmatrix} \widetilde \Sigma & 0 \\ 0 & \widetilde \Sigma_{k-1,k-1} \end{bmatrix}, \quad  V = \begin{bmatrix} \widetilde V & v \end{bmatrix},
\]
$\widetilde X= \widetilde U \widetilde \Sigma \widetilde V^T$ is the local optimizer for rank $k-1$ and $u$ and $v$ are random unit-norm vectors orthogonal to $U$ and $V$, respectively. We call this the \emph{homotopy strategy} which is also used in one of LMAFit's rank adaptivity strategies (but with zero vectors $u$ and $v$). The error using this homotopy strategy is labeled ``(hom)'' in Figure~\ref{fig:rank_decay} and it clearly performs much more favorably. In addition, LRGeomCG is able to obtain a smaller error.

\section{Conclusions}

The matrix completion consists of recovering a low-rank matrix based on a very sparse set of entries of this matrix. In this paper, we have presented a new method based on optimization on manifolds to solve large-scale matrix completion problems. Compared to most other existing methods in the literature, our approach consists of directly minimizing the least-square error of the fit directly over $\Manifold$, the set of matrices of rank $k$.

The main contribution of this paper is to show that the lack of vector space structure of $\Manifold$ does not need be an issue when optimizing over this set, and illustrating this for low-rank matrix completion. Using the framework of retraction-based optimization, the necessary expressions were derived in order to minimize any smooth objective function over the Riemannian manifold $\Manifold$. The geometry chosen for $\Manifold$, namely a submanifold embedded in the space of matrices, allowed for an efficient implementation of non-linear CG for matrix completion. Indeed, the numerical experiments illustrated that this approach is very competitive with state-of-the-art solvers for matrix completion. In particular, the method achieved very fast asymptotic convergence factors without any tuning of parameters thanks to a simple computation of the initial guess for the line search.

A drawback of the proposed approach is however that the rank of the manifold is fixed. Although there are several problems where choosing the rank is straightforward, an integrated manifold-based solution, like in \cite{Cason2011, Mishra:2011b}, is desirable.  In addition, our convergence proof relied on several safeguards and modifications in the algorithm that seem completely unnecessary in practice. Closing these gaps are currently topics of further research.

\section*{Acknowledgments}

Parts of this paper have been prepared while the author was affiliated with the Seminar of Applied Mathematics, ETH Z\"urich. In addition, he gratefully acknowledges the helpful discussions with Daniel Kressner regarding low-rank matrix completion.

\appendix

\section{Proof of Proposition~\ref{prop:Hessian}}\label{app}

The proof of Prop.~\ref{prop:Hessian} consists of two steps. First, we introduce a second-order retraction which allows us to use standard Euclidean derivatives to derive the Riemannian Hessian of any objective function. Then, we derive a second-order model of the specific function $f$ based on a second-order Taylor expansion using this retraction. This second-order model will give us the Riemannian Hessian.

\subsection{An explicit second-order retraction on $\Manifold$}

The Riemannian Hessian of an objective function $f$ is usually defined in terms of the Levi-Civita connection, but in case of an embedded submanifold it can also be defined by means of so-called second-order retractions. 

\emph{Second-order retractions} are defined as second-order approximations of the exponential map \cite[Proposition 3]{Absil:2012}. In case of embedded submanifolds, it is shown in \cite{Absil:2012} that they can be identified as retractions that additionally satisfy
\begin{equation}\label{eq:deriv_zero}
  \proj_{T_X \Manifold} \left( \frac{\deriv}{\deriv t^2}  R_X(t \xi) \Big|_{t=0} \right) = 0, \quad \text{for all $\xi \in T_X \Manifold$}.
\end{equation}
Now, the Riemannian Hessian operator, a symmetric and linear operator 
\[
  \Hess f(X): T_X \Manifold \to T_X \Manifold,
\]
can be determined from the classical Hessian of the function $f \circ R_X:  T_X \Manifold \to \RSpace$ defined on a vector space since
\[
 \Hess f(X) = \Hess (f \circ R_X)(0),
\]
where $R_X$ is a second-order retraction, see \cite[Proposition 5.5.5]{AMS2008}.

The next proposition gives a second-order retraction and the plan of this construction is as follows. First, we construct a smooth mapping $T_X\Manifold \to \RSpace^{m \times n}$; then, we prove that it satisfies the requirements in Definition~\ref{def:retr} and \eqref{eq:deriv_zero}; and finally, we expand in series. Since the metric projection \eqref{eq:R_X} can be shown to satisfy \eqref{eq:deriv_zero} also, we have in addition obtained a second-order Taylor series for \eqref{eq:R_X}.

\begin{prpstn}\label{prop:R2}
For any $X=U \Sigma V^T \in \Manifold$ satisfying \eqref{eq:x_USV} and $\xi \in T_X\Manifold$ satisfying \eqref{eq:Tspace_short}, the mapping $\soR_X$, given by
\begin{equation}\label{eq:R2}
 \soR_X : \TSpace_X\Manifold \rightarrow \Manifold, \quad \xi = UMV^T + U_p V^T + UV_p^T  \mapsto Z_U Z_V^T,
\end{equation}
where
\begin{align}
Z_U & \defeq U (\Sigma + \tfrac{1}{2} M - \tfrac{1}{8}M\Sigma^{-1}M)+ U_p(I_k - \tfrac{1}{2}\Sigma^{-1}M), \label{eq:ZU} \\
Z_V & \defeq V(I_k + \tfrac{1}{2}M^T\Sigma^{-T} - \tfrac{1}{8}M^T\Sigma^{-T}M^T\Sigma^{-T})  + V_p(\Sigma^{-T} - \tfrac{1}{2}\Sigma^{-T}M^T\Sigma^{-T}), \label{eq:ZV}
\end{align}
is a second-order approximation of the exponential map on $(g_X,\Manifold)$. Furthermore, we have
\begin{equation}\label{eq:R2_expansion}
 \soR_X(\xi) = X + \xi + U_p \Sigma^{-1} V_p^T + O(\norm{\xi}^3), \quad \norm{\xi} \to 0,
\end{equation}
and
\begin{equation}\label{eq:R2_Rx_error}
 \soR_X(\xi) - R_X(\xi) =  O(\norm{\xi}^3), \quad \norm{\xi} \to 0,
\end{equation}
with $R_X$ the metric projection \eqref{eq:R_X}.
\end{prpstn}

Before proving Proposition~\ref{prop:R2}, we first state the following lemma that shows smoothness of $\soR_X$ in $\xi$ for \emph{fixed} $X$.
\begin{lmm}\label{lem:R2_smooth}
  Using the notation of Proposition~\ref{prop:R2}, we have
  \begin{equation}\label{eq:R2_wxw}
   \soR_X(\xi) = W X^\dagger W,
  \end{equation}
  with
  \begin{align*}
   W &\defeq X + \tfrac{1}{2} \proj_X^{\textrm{s}}(\xi) + \proj_X^{\textrm{p}}(\xi) \\
 & \quad- \tfrac{1}{8} \proj_X^{\textrm{s}}(\xi) X^\dagger \proj_X^{\textrm{s}}(\xi) - \tfrac{1}{2} \proj_X^{\textrm{p}}(\xi) X^\dagger \proj_X^{\textrm{s}}(\xi) - \tfrac{1}{2} \proj_X^{\textrm{s}}(\xi) X^\dagger \proj_X^{\textrm{p}}(\xi),
  \end{align*}
and
\begin{align}
  \proj_X^{\textrm{s}} &\colon \RSpace^{n \times m} \to \TSpace_X\Manifold \colon Z \mapsto \proj_U Z \proj_V, \label{eq:P1a}  \\
  \proj_X^{\textrm{p}} &\colon \RSpace^{n \times m} \to \TSpace_X\Manifold \colon Z \mapsto \proj_U^\perp Z \proj_V + \proj_U Z \proj_V^\perp. \label{eq:P2a}
\end{align}
  Furthermore, for fixed $X \in \Manifold$, the mapping $\soR_X(\xi)$ is smooth in $\xi$ as a mapping from $T_X \Manifold$ to $\RSpace^{m \times n}$. 
\end{lmm}
\begin{proof}
Substituting  $X^\dagger = V\Sigma^{-1}U^T$ and the definitions \eqref{eq:P1a},\eqref{eq:P2a} for $\proj_X^{\textrm{s}},\proj_X^{\textrm{p}}$ in $W$ gives
\begin{align*}
 W &= U \Sigma V^T + \tfrac{1}{2} UMV^T + U_pV^T+V V_p^T - \tfrac{1}{8} UMV^TV\Sigma^{-1}U^TUMV^T \\
   &\phantom{=\ } - \tfrac{1}{2} (U_pV^T + U V_p^T) V\Sigma^{-1}U^TUMV^T - \tfrac{1}{2} UMV^TV\Sigma^{-1}U^TU_pV^T + U V_p^T, \\
   &= U \Sigma V^T + \tfrac{1}{2} UMV^T + U_pV^T+U V_p^T - \tfrac{1}{8} UM\Sigma^{-1}M  V^T \\
   &\phantom{=\ }- \tfrac{1}{2} U_p \Sigma^{-1}MV^T - \tfrac{1}{2} U M \Sigma^{-1} V_p^T.
\end{align*}
This shows the equivalence of \eqref{eq:R2_wxw} to \eqref{eq:R2} after some tedious but straightforward algebraic manipulations. 

Observe that \eqref{eq:R2_wxw} consists of only smooth operations involving $X$: taking transposes, multiplying matrices and inverting full rank matrices can all be expressed as rational functions without singularities in the matrix entries. From this, we have immediately smoothness of $\soR_X(\xi)$ in $\xi$. 
\qquad\end{proof}    

The previous lemma shows smoothness of $\soR_X: T_X\Manifold \to \RSpace^{m \times n}$. For Proposition~\ref{prop:R2}, we also need to prove local smoothness of $\soR$ as a function from the tangent bundle onto $\Manifold$. We will do this by relying on Boman's theorem \cite[Theorem 1]{Boman:1967}, which states that a function $f: \RSpace^d \to \RSpace$ is smooth if $f \circ \gamma$ is smooth along all smooth curves $\gamma: \RSpace \to \RSpace^d$.

In addition, for the proof below we will allow the matrix $\Sigma$ in definition \eqref{eq:R2} to be any full-rank matrix. This is possible because none of the derivations above relied on $\Sigma$ being diagonal and definition \eqref{eq:R2} for $\soR_X$ is independent of the choice of specific matrices $U,V,\Sigma$ in the factorization of $X$. Indeed, suppose that we would have chosen $\widetilde{U} = UQ_U$, $\widetilde{V} = VQ_V$ and $\widetilde{\Sigma} = Q_U^T \Sigma Q_V$ as factorization matrices. Then the coefficient matrices of the tangent vector $\xi$ would have been $\widetilde{M} = Q_U^T M Q_V$, $\widetilde{U}_p = U_pQ_V$ and $\widetilde{V} = V_pQ_U$. Now, it is straightforward to check that the corresponding expressions for \eqref{eq:ZU}--\eqref{eq:ZV} become $Z_{\widetilde{U}} = Z_U Q_V$ and $Z_{\widetilde{V}} = Z_V Q_V$. Hence, $\soR_X(\xi) = Z_{\widetilde{U}} Z_{\widetilde{V}}^T = Z_U Z_V^T$ stays the same.

{\em Proof of Proposition~\ref{prop:R2}.}
  First, we show \eqref{eq:R2_expansion}. Elementary manipulation using \eqref{eq:ZU}--\eqref{eq:ZV} reveals that
\begin{align*}
  \soR_X(\xi) 
   &= \begin{bmatrix} U & U_p \end{bmatrix} 
   \begin{bmatrix} \Sigma + M + O(\norm{M}^3) & I +  O(\norm{M}^2) \\ I +  O(\norm{M}^2) & \Sigma^{-1} + O(\norm{M}) \end{bmatrix} 
   \begin{bmatrix} V & V_p \end{bmatrix}^T, \quad \norm{M} \to 0.
  \end{align*}
  For fixed $X$, the norms $\norm{M},\norm{U_p}, \norm{V_p}$ are all $O(\norm{\xi})$, and so
  \[
  \soR_X(\xi) = U(\Sigma + M) V^T + U V_p^T + U_p V^T + U_p \Sigma^{-1} V_p ^T + O(\norm{\xi}^3), \quad \norm{\xi} \to 0,
  \]
  which can be rewritten as \eqref{eq:R2_expansion}. Next, we show that $\soR_X$ is a retraction according to Definition~\ref{def:retr}. By construction, the image of $R_X$ is a matrix $Z_{U}Z_{V}^T$ of rank at most $k$. Since  $\rank(X)$ is continuous in $X$, one can take $\delta_t > 0$ sufficiently small such that the image of $t\mapsto \soR_X(t\xi)$ is always of rank $k$ for all $t \in (-\delta_t, \delta_t)$. 

Finally, we establish smoothness of $\soR_X$ in a neighborhood of $(X,0) \in \TSpace\Manifold$. Take any $(Y,\xi) \in T\Manifold$ sufficiently close to $(X,0)$ and consider any smooth curve $\gamma: \RSpace \to T\Manifold$ that connects $(X,0)$ to $(Y,\xi)$. Consider the partitioning $\gamma(t) = (\alpha(t),\beta(t))$ with $\alpha$ a smooth curve on $\Manifold$ and $\beta$ on $T_{\alpha(t)}\Manifold$. By the existence of a smooth SVD for fixed-rank matrices \cite[Theorem 2.4]{Chern:2000}, there exists smooth $U(t), V(t), \Sigma(t)$ (like mentioned above, $\Sigma(s)$ is now a full-rank matrix) such that $\alpha(t) = U(t) \Sigma(t) V(t)^T$. Furthermore by Lemma~\ref{lem:R2_smooth} map $\soR_{\alpha(t)}(\xi)$ is smooth in $\xi$ around zero, so it will be smooth along $\soR_{\alpha(t)}(\beta(t))$. Then, the expression $R_{\alpha(t)}(\beta(t))$ consists of smooth matrix operations that depend smoothly on $t$. Since the curve $\gamma(t) = (\alpha(t),\beta(t))$ was chosen arbitrarily, we have shown smoothness of $R_X(\xi)$ on the tangent bundle in a neighborhood of $(X,0)$. To conclude, observe from \eqref{eq:R2_expansion} that $\soR_X$ obeys \eqref{eq:deriv_zero} since $U_p \Sigma^{-1} V_p^T$ is orthogonal to the tangent space. This makes $\soR_X$ a second-order retraction.
\qquad\endproof

Although parts of the previous results can also be found in \cite[Theorem 1]{Shalit:2010}, we prove additionally that $\soR_X$ satisfies Definition~\ref{def:retr} by showing smoothness on the whole tangent bundle (and not only for each tangent space separately, as in \cite{Shalit:2010}).

\subsection{Second-order model of $f$ on $\Manifold$}

Based on the Taylor series of $\soR_X$, it is now straightforward to derive the second-order model
\[
 m(X) = f(X) + \innprod{\Grad f(X)}{\xi} + \frac{1}{2}\innprod{\Hess f(X)[\xi]}{\xi}
\]
of any smooth objective function $f: \Manifold \to \RSpace$. In this section, we will illustrate this with our objective function $f(X) = \frac{1}{2}\normF{\proj_\Omega(X-A)}^2$.

Expanding $f(\soR_X(\xi))$ using \eqref{eq:R2_expansion} gives for $\norm{\xi} \to 0$
\begin{align*}
 f(\soR_X(\xi)) &= \frac{1}{2} \normF{ \proj_\Omega(\soR_X(\xi)-A) }^2 \\
 &= \frac{1}{2} \normF{ \proj_\Omega(X+\xi + U_p \Sigma^{-1} V_p^T-A) }^2 + O(\norm{\xi}^3) \\
 &= \frac{1}{2} \Tr[ \proj_\Omega(X-A)^T \proj_\Omega(X-A) + 2 \proj_\Omega(X-A)^T \proj_\Omega(\xi) \\
 &\quad + 2 \proj_\Omega(X-A)^T \proj_\Omega( U_p \Sigma^{-1} V_p^T ) + \proj_\Omega(\xi)^T \proj_\Omega(\xi) ] + O(\norm{\xi}^3).
\end{align*}
We can recognize the constant term 
\[
 f(X) = \frac{1}{2} \Tr[ \proj_\Omega(X-A)^T \proj_\Omega(X-A) ],
\]
the linear term
\begin{equation}\label{eq:id_grad}
  \innprod{\Grad f(X)}{\xi} = \Tr[ \proj_\Omega(X-A)^T \proj_\Omega(\xi) ],
\end{equation}
and the quadratic term
\begin{equation}\label{eq:id_hess}
 \innprod{\Hess f(X)[\xi]}{\xi} =  \Tr[ 2 \proj_\Omega(X-A)^T \proj_\Omega( U_p \Sigma^{-1} V_p^T ) + \proj_\Omega(\xi)^T \proj_\Omega(\xi) ].
\end{equation}

Using  \eqref{eq:riem_gradient}, we see that our expression for the Riemannian gradient, $ \Grad f(x) = \projTxM (\proj_\Omega(X-A))$, indeed satisfies \eqref{eq:id_grad}. This only leaves determining the Riemannian Hessian as the symmetric and linear mapping
\[
 \Hess f(X) : \TSpace_X\Manifold \rightarrow \TSpace_X\Manifold, \quad \xi \mapsto \Hess f(X)[\xi],
\]
that satisfies the inner product \eqref{eq:id_hess}. First note that $U_p \Sigma^{-1} V_p^T$ is quadratic in $\xi$ since
\[
 U_p \Sigma^{-1} V_p^T = (U_p V^T + UV_p^T) (V \Sigma^{-1} U^T) (U_p V^T + UV_p^T) = \proj_X^{\textrm{p}}(\xi) X^\dagger \proj_X^{\textrm{p}}(\xi).
\]
Now we get
\begin{align*}
 \innprod{\Hess f(X)[\xi]}{\xi} &=  2\innprod{\proj_\Omega(X-A)}{\proj_\Omega( \proj_X^{\textrm{p}}(\xi) X^\dagger \proj_X^{\textrm{p}}(\xi) )} + \innprod{\proj_\Omega(\xi)}{\proj_\Omega(\xi)} \\
&=  \underbrace{2\innprod{\proj_\Omega(X-A)}{\proj_X^{\textrm{p}}(\xi) X^\dagger \proj_X^{\textrm{p}}(\xi)}}_{f_1:=\innprod{\calH_1(\xi)}{\xi}} + \underbrace{\innprod{\proj_\Omega(\xi)}{\xi}.}_{f_2:=\innprod{\calH_2(\xi)}{\xi}}
\end{align*}
So, the inner product with the Hessian consists of two parts, $f_1$ and $f_2$, and the Hessian itself is the sum of two operators, $\mathcal{H}_1$ and $\mathcal{H}_2$. The second contribution to the Hessian is readily available from $f_2$:
\[
 f_2 = \innprod{\proj_\Omega(\xi)}{\xi} = \innprod{\proj_\Omega(\xi)}{\projTxM(\xi)} = \innprod{\projTxMp \proj_\Omega(\xi)}{\xi},
\]
and so 
\[
 \calH_2(\xi) := \projTxM \proj_\Omega(\xi) = \proj_U \proj_\Omega(\xi) \proj_V + \proj_U^\perp \proj_\Omega(\xi) \proj_V + \proj_U \proj_\Omega(\xi) \proj_V^\perp.
\]
In $f_1$ we still need to separate $\xi$ to one side of the inner product. Since we can choose whether we bring over the first $\projTxMp(\xi)$ or the second, we introduce a constant $c \in \RSpace$,
\begin{align*}
 f_1 &= 2 c \ \innprod{\proj_\Omega(X-M) \cdot (X^\dagger \projTxMp(\xi))^T}{\projTxMp(\xi)} \\
  &\quad + 2 (1-c) \ \innprod{(\projTxMp(\xi) X^\dagger )^T \cdot \proj_\Omega(X-A)}{\projTxMp(\xi)} \\
  &= 2 c \ \innprod{\projTxMp [ \proj_\Omega(X-A) \cdot \projTxMp(\xi)^T \cdot (X^\dagger)^T ] }{\xi} \\
  &\quad + 2 (1-c) \ \innprod{\projTxMp [ (X^\dagger)^T \cdot \projTxMp(\xi)^T \cdot \proj_\Omega(X-A) ]}{\xi}. 
\end{align*}
The operator $\calH_1$ is clearly linear. Imposing the symmetry requirement $\innprod{\calH_1(\xi)}{\nu} = \innprod{\xi}{\calH_1(\nu)}$ for any arbitrary tangent vector $\nu$, we get that $c=1/2$, and so
\begin{align*} 
 \calH_1(\xi) &=  \projTxMp [ \proj_\Omega(X-A) \projTxMp(\xi)^T (X^\dagger)^T +  (X^\dagger)^T \projTxMp(\xi)^T \proj_\Omega(X-A) ] \\
               &= \proj_U^\perp \proj_\Omega(X-A) V_p \Sigma^{-1} V^T + U \Sigma^{-1}U_p^T\proj_\Omega(X-A)\proj_V^\perp.
\end{align*}
 Finally, we can put together the whole Hessian
\begin{align*}
 \Hess f(X)[\xi] &= \proj_U \proj_\Omega(\xi) \proj_V  + \proj_U^\perp (\proj_\Omega(\xi) + \proj_\Omega(X-A) V_p \Sigma^{-1} V^T )  \proj_V \\
 &\quad + \proj_U (\proj_\Omega(\xi) + U \Sigma^{-1}U_p^T\proj_\Omega(X-A))\proj_V^\perp,
\end{align*}
and this also proves Prop.~\ref{prop:Hessian}.

\bibliographystyle{siam}

\bibliography{bibexport}

\end{document}

%% file: manifold_optim.pdf_tex

\begingroup
  \makeatletter
  \providecommand\color[2][]{%
    \errmessage{(Inkscape) Color is used for the text in Inkscape, but the package 'color.sty' is not loaded}
    \renewcommand\color[2][]{}%
  }
  \providecommand\transparent[1]{%
    \errmessage{(Inkscape) Transparency is used (non-zero) for the text in Inkscape, but the package 'transparent.sty' is not loaded}
    \renewcommand\transparent[1]{}%
  }
  \providecommand\rotatebox[2]{#2}
  \ifx\svgwidth\undefined
    \setlength{\unitlength}{473.25653681pt}
  \else
    \setlength{\unitlength}{\svgwidth}
  \fi
  \global\let\svgwidth\undefined
  \makeatother
  \begin{picture}(1,0.66758515)%
    \put(0,0){\includegraphics[width=\unitlength]{manifold_optim.pdf}}%
    \put(0.3298679,0.23627405){\color[rgb]{0.07058824,0.52156863,0}\makebox(0,0)[lb]{\smash{$t \mapsto R_{X_i}(t \eta_i)$}}}%
    \put(0.7980119,0.45906199){\color[rgb]{0.16862745,0.16862745,0.85098039}\makebox(0,0)[lb]{\smash{$T_{X_i} \Manifold$}}}%
    \put(0.19485858,0.00298134){\color[rgb]{0,0,0}\makebox(0,0)[lb]{\smash{$\Manifold$}}}%
    \put(0.30252373,0.47121144){\color[rgb]{0,0,0}\makebox(0,0)[lb]{\smash{$X_i$}}}%
    \put(0.54400985,0.09196044){\color[rgb]{0,0,0}\makebox(0,0)[lb]{\smash{$X_{i+1} = R_{X_i}(0.5^m t_i \eta_i)$}}}%
    \put(0.54111678,0.60257887){\color[rgb]{0.16862745,0.16862745,0.85098039}\makebox(0,0)[lb]{\smash{$\eta_i$}}}%
    \put(0.29802606,0.40295729){\color[rgb]{0.16862745,0.16862745,0.85098039}\makebox(0,0)[lb]{\smash{$-\xi_i$}}}%
    \put(0.401307,0.65217332){\color[rgb]{0.16862745,0.16862745,0.85098039}\makebox(0,0)[lb]{\smash{$\beta_i \mathcal{T}_{X_{i-1} \to X_{i}}(\eta_{i-1})$}}}%
  \end{picture}%
\endgroup

%% file: manifold_vector_transport.pdf_tex

\begingroup
  \makeatletter
  \providecommand\color[2][]{%
    \errmessage{(Inkscape) Color is used for the text in Inkscape, but the package 'color.sty' is not loaded}
    \renewcommand\color[2][]{}%
  }
  \providecommand\transparent[1]{%
    \errmessage{(Inkscape) Transparency is used (non-zero) for the text in Inkscape, but the package 'transparent.sty' is not loaded}
    \renewcommand\transparent[1]{}%
  }
  \providecommand\rotatebox[2]{#2}
  \ifx\svgwidth\undefined
    \setlength{\unitlength}{590.82811884pt}
  \else
    \setlength{\unitlength}{\svgwidth}
  \fi
  \global\let\svgwidth\undefined
  \makeatother
  \qquad \qquad \qquad \qquad \begin{picture}(1,0.49635295)%
    \put(0,0){\includegraphics[width=\unitlength]{manifold_vector_transport.pdf}}%
    \put(0.63265808,0.19169548){\color[rgb]{0.0745098,0.52156863,0}\makebox(0,0)[lb]{\smash{$\mathcal{T}_{X \to Y}(\xi)$}}}%
    \put(0.15608278,0.00238807){\color[rgb]{0,0,0}\makebox(0,0)[lb]{\smash{$\Manifold$}}}%
    \put(0.23670562,0.39429549){\color[rgb]{0,0,0}\makebox(0,0)[lb]{\smash{$X$}}}%
    \put(0.17708998,0.16926247){\color[rgb]{0,0,0}\makebox(0,0)[lb]{\smash{$Y = R_{X}(\eta)$}}}%
    \put(0.30142572,0.33380464){\color[rgb]{0.16862745,0.16862745,0.85098039}\makebox(0,0)[lb]{\smash{$\eta$}}}%
    \put(0.33736541,0.48400799){\color[rgb]{0.0745098,0.52156863,0}\makebox(0,0)[lb]{\smash{$\xi $}}}%
    \put(0.41170224,0.08683509){\color[rgb]{0.16862745,0.16862745,0.85098039}\makebox(0,0)[lb]{\smash{$T_{Y} \Manifold$}}}%
    \put(0.06716092,0.43793014){\color[rgb]{0.16862745,0.16862745,0.85098039}\makebox(0,0)[lb]{\smash{$T_{X} \Manifold$}}}%
  \end{picture}%
\endgroup

%% file: MatrixCompletionManifold-ext-report.bbl
\begin{thebibliography}{10}

\bibitem{Absil:2007}
{\sc P.-A. Absil, C.~G. Baker, and K.~A. Gallivan}, {\em Trust-region methods
  on {R}iemannian manifolds}, Found. Comput. Math., 7 (2007), pp.~303--330.

\bibitem{AMS2008}
{\sc P.-A. Absil, R.~Mahony, and R.~Sepulchre}, {\em Optimization Algorithms on
  Matrix Manifolds}, Princeton University Press, 2008.

\bibitem{Absil:2010}
{\sc P.-A. Absil and J.~Malick}, {\em Projection-like retractions on matrix
  manifolds}, SIAM J. Optim., 22 (2012), pp.~135--158.

\bibitem{Absil:2012}
\leavevmode\vrule height 2pt depth -1.6pt width 23pt, {\em Projection-like
  retractions on matrix manifolds}, SIAM J. Optim., 22 (2012), pp.~135--158.

\bibitem{Netflix:07}
{\sc {ACM SIGKDD and Netflix}}, ed., {\em Proceedings of {KDD Cup and
  Workshop}}, 2007.

\bibitem{Adler:2002}
{\sc R.~L. Adler, J.-P. Dedieu, J.~Y. Margulies, M.~Martens, and M.~Shub}, {\em
  Newton's method on {R}iemannian manifolds and a geometric model for the human
  spine}, IMA J. Numer. Anal., 22 (2002).

\bibitem{Balzano:2010}
{\sc L.~Balzano, R.~Nowak, and B.~Recht}, {\em Online identification and
  tracking of subspaces from highly incomplete information}, in Proceedings of
  Allerto, 2010.

\bibitem{Bertsekas:1999}
{\sc D.~P. Bertsekas}, {\em Nonlinear Programming}, Athena Scientific, 1999.

\bibitem{Boman:1967}
{\sc J.~Boman}, {\em Differentiability of a function and of its compositions
  with functions of one variable}, Math. Scand., 20 (1967), pp.~249---268.

\bibitem{Boumal:2011}
{\sc N.~Boumal and P.-A. Absil}, {\em {RTRMC}: A {R}iemannian trust-region
  method for low-rank matrix completion}, in Proceedings of the Neural
  Information Processing Systems Conference (NIPS), 2011.

\bibitem{Boumal:2012fk}
\leavevmode\vrule height 2pt depth -1.6pt width 23pt, {\em Low-rank matrix
  completion via trust-regions on the grassmann manifold}, Tech. Report
  2012.07, Universit{\'e} catholique de Louvain, INMA, 2012.

\bibitem{Bruns:1988}
{\sc W.~Bruns and U.~Vetter}, {\em Determinantal rings}, vol.~1327 of Lecture
  Notes in Mathematics, Springer-Verlag, Berlin, 1988.

\bibitem{Burer:2003}
{\sc S.~Burer and R.~D.~C. Monteiro}, {\em A nonlinear programming algorithm
  for solving semidefinite programs via low-rank factorization}, Math.
  Program., 95 (2003), pp.~329--357.

\bibitem{BurMon2005}
\leavevmode\vrule height 2pt depth -1.6pt width 23pt, {\em Local minima and
  convergence in low-rank semidefinite programming}, Math. Program., 103
  (2005), pp.~427--444.

\bibitem{Cai:2010}
{\sc J.-F. Cai, E.~J. Cand{\`e}s, and Z.~Shen}, {\em A singular value
  thresholding algorithm for matrix completion}, SIAM J. Optim., 20 (2010),
  pp.~1956--1982.

\bibitem{Candes:2009b}
{\sc E.~Cand{\`e}s and B.~Recht}, {\em Exact matrix completion via convex
  optimization}, Found. Comput. Math., 9 (2009), pp.~717--772.

\bibitem{Candes:2010}
{\sc E.~J. Cand{\'e}s and Y.~Plan}, {\em Matrix completion with noise},
  Proceedings of the IEEE, 98 (2010), pp.~925--936.

\bibitem{Candes:2009}
{\sc E.~J. Cand{\`e}s and T.~Tao}, {\em The power of convex relaxation:
  {N}ear-optimal matrix completion}, IEEE Trans. Inform. Theory, 56 (2009),
  pp.~2053--2080.

\bibitem{Cason2011}
{\sc T.~P. Cason, P.-A. Absil, and P.~{Van Dooren}}, {\em Iterative methods for
  low rank approximation of graph similarity matrices}, Lin. Alg. Appl.,
  (2011).

\bibitem{Chern:2000}
{\sc J.~Chern and L.~Dieci}, {\em Smoothness and periodicity of some matrix
  decompositions}, SIAM J. Matrix Anal. Appl., 22 (2000).

\bibitem{Dai:2010}
{\sc W.~Dai and O.~Milenkovic}, {\em {SET}: an algorithm for consistent matrix
  completion}, in International Conference on Acoustics, Speech, and Signal
  Processing (ICASSP), 2010.

\bibitem{Dai:2011}
{\sc W.~Dai, O.~Milenkovic, and E.~Kerman}, {\em Subspace evolution and
  transfer ({SET}) for low-rank matrix completion}, IEEE Trans. Signal
  Process., 59 (2011), pp.~3120--3132.

\bibitem{Edelman:1999}
{\sc A.~Edelman, T.~A. Arias, and S.~T. Smith}, {\em The geometry of algorithms
  with orthogonality constraints}, SIAM J. Matrix Anal. Appl., 20 (1999),
  pp.~303--353.

\bibitem{Goldfarb:2011}
{\sc D.~Goldfarb and S.~Ma}, {\em Convergence of fixed point continuation
  algorithms for matrix rank minimization}, Found. Comput. Math., 11 (2011),
  pp.~183--210.

\bibitem{Golub:1996}
{\sc G.~H. Golub and C.~F.~Van Loan}, {\em Matrix Computations}, Johns Hopkins
  Studies in Mathematical Sciences, 3rd~ed., 1996.

\bibitem{Golub:1973}
{\sc G.~H. Golub and V.~Pereyra}, {\em The differentiation of pseudo-inverses
  and nonlinear least squares problems whose variables separate}, SIAM Journal
  on Numerical Analysis, 10 (1973).

\bibitem{Harbrecht:2010}
{\sc H.~Harbrecht, M.~Peters, and R.~Schneider}, {\em On the low-rank
  approximation by the pivoted {C}holesky decomposition.}, Applied Numerical
  Mathematics, 62 (2012), pp.~428--440.

\bibitem{HM94}
{\sc U.~Helmke and J.~B. Moore}, {\em Optimization and Dynamical Systems},
  Springer-Verlag London Ltd., 1994.

\bibitem{Journee:2010}
{\sc M.~Journ{\'e}e, F.~Bach, P.-A. Absil, and R.~Sepulchre.}, {\em Low-rank
  optimization on the cone of positive semidefinite matrices}, SIAM J. Optim.,
  20 (2010), pp.~2327--2351.

\bibitem{Keshavan:2010a}
{\sc R.~Keshavan, A.~Montanari, and S.~Oh}, {\em Matrix completion from noisy
  entries}, JMLR, 11 (2010), pp.~2057--2078.

\bibitem{Keshavan:2010}
{\sc R.~H. Keshavan, A.~Montanari, and S.~Oh}, {\em Matrix completion from a
  few entries}, IEEE Trans. Inform. Theory, 56 (2010), pp.~2980--2998.

\bibitem{Koch:2007}
{\sc O.~Koch and C.~Lubich}, {\em Dynamical low-rank approximation}, SIAM. J.
  Matrix Anal., 29 (2007), pp.~434--454.

\bibitem{Kressner:2010c}
{\sc D.~Kressner and C.~Tobler}, {\em {Krylov subspace methods for linear
  systems with tensor product structure}}, SIAM J. Matrix Anal. Appl., 31
  (2010), pp.~1688--1714.

\bibitem{Larsen:2004}
{\sc R.~M. Larsen}, {\em {PROPACK}---software for large and sparse {SVD}
  calculations}.
\newblock \url{http://soi.stanford.edu/~rmunk/PROPACK}, 2004.

\bibitem{Lee2003}
{\sc John~M. Lee}, {\em Introduction to smooth manifolds}, vol.~218 of Graduate
  Texts in Mathematics, Springer-Verlag, New York, 2003.

\bibitem{Lee:2010}
{\sc K.~Lee and Y.~Bresler}, {\em {ADMiRA}: Atomic decomposition for minimum
  rank approximation}, IEEE Trans. Inform. Theory, 56 (2010), pp.~4402--4416.

\bibitem{Lewis2008}
{\sc A.~S. Lewis and J.~Malick}, {\em Alternating projections on manifolds},
  Math. Oper. Res., 33 (2008), pp.~216--234.

\bibitem{Lin:2009a}
{\sc Z.~Lin, M.~Chen, L.~Wu, and Yi~Ma}, {\em The augmented {L}agrange
  multiplier method for exact recovery of corrupted low-rank matrices,}, Tech.
  Report UILU-ENG-09-2215, University of Illinois, Urbana, Department of
  Electrical and Computer Engineering, 2009.

\bibitem{Liu:2010}
{\sc Y.-J. Liu, D.~Sun, and K.-C. Toh}, {\em An implementable proximal point
  algorithmic framework for nuclear norm minimization}, Math. Program., 133
  (2012), pp.~399--436.

\bibitem{Ma:2011}
{\sc S.~Ma, D.~Goldfarb, and L.~Chen}, {\em Fixed point and {B}regman iterative
  methods for matrix rank minimization}, Math. Progr., 128 (2011),
  pp.~321--353.

\bibitem{Meka:2010}
{\sc R.~Meka, P.~Jain, and I.~S. Dhillon}, {\em Guaranteed rank minimization
  via singular value projection}, in Proceedings of the Neural Information
  Processing Systems Conference (NIPS), 2010.

\bibitem{Meyer:2011a}
{\sc G.~Meyer}, {\em Geometric optimization algorithms for linear regression on
  fixed-rank matrices}, PhD thesis, University of Li{\`e}ge, 2011.

\bibitem{Meyer:2011}
{\sc G.~Meyer, S.~Bonnabel, and R.~Sepulchre}, {\em Linear regression under
  fixed-rank constraints: a {R}iemannian approach}, in Proc. of the 28th
  International Conference on Machine Learning (ICML2011), Bellevue (USA),
  2011.

\bibitem{Meyer:2011c}
\leavevmode\vrule height 2pt depth -1.6pt width 23pt, {\em Regression on
  fixed-rank positive semidefinite matrices: a {R}iemannian approach}, Journal
  of Machine Learning Research, 12 (2011), pp.~593--−625.

\bibitem{Michenkova:2011}
{\sc M.~Michenkov{\'a}}, {\em Numerical algorithms for low-rank matrix
  completion problems}.
\newblock \url{http://www.math.ethz.ch/~kressner/students/michenkova.pdf},
  2011.

\bibitem{Mishra:2011b}
{\sc B.~Mishra, G.~Meyer, F.~Bach, and R.~Sepulchre}, {\em Low-rank
  optimization with trace norm penalty}, Pre-print,  (2011).
\newblock http://arxiv.org/abs/1112.2318.

\bibitem{Mishra:2012fk}
{\sc B.~Mishra, G.~Meyer, S.~Bonnabel, and R.~Sepulchre}, {\em Fixed-rank
  matrix factorizations and riemannian low-rank optimization}, Arxiv,  (2012).

\bibitem{Orsi:2006}
{\sc R.~Orsi, U.~Helmke, and J.~B. Moore}, {\em A {N}ewton-like method for
  solving rank constrained linear matrix inequalities}, Automatica, 42 (2006),
  pp.~1875--1882.

\bibitem{Qi:2010}
{\sc C.~Qi, K.~A. Gallivan, and P.-A. Absil}, {\em Riemannian {BFGS} algorithm
  with applications}, in Recent Advances in Optimization and its Applications
  in Engineering, 2010.

\bibitem{Shalit:2010}
{\sc U.~Shalit, D.~Weinshall, and G.~Chechik}, {\em Online learning in the
  manifold of low-rank matrices}, in Neural Information Processing Systems
  (NIPS spotlight), 2010.

\bibitem{Shalit:2012}
\leavevmode\vrule height 2pt depth -1.6pt width 23pt, {\em Online learning in
  the embedded manifold of low-rank matrices}, Journal of Machine Learning
  Research, 13 (2012), pp.~429--458.

\bibitem{Smith:1994}
{\sc S.~T. Smith}, {\em Optimization techniques on {R}iemannian manifold}, in
  Hamiltonian and Gradient Flows, Algorithms and Control, vol.~3, Amer. Math.
  Soc., Providence, RI, 1994, pp.~113---136.

\bibitem{Toh:2010}
{\sc K.~Toh and S.~Yun}, {\em An accelerated proximal gradient algorithm for
  nuclear norm regularized least squares problems}, Pacific J. Optimization,
  (2010).

\bibitem{Vandereycken:2012d}
{\sc B.~Vandereycken}, {\em Low-rank matrix completion by riemannian
  optimization}, tech. report, ANCHP-MATHICSE, Mathematics Section, {'E}cole
  Polytechnique F{'e}d{'e}rale de Lausanne, 2011.

\bibitem{Vandereycken2010}
{\sc B.~Vandereycken and S.~Vandewalle}, {\em A {R}iemannian optimization
  approach for computing low-rank solutions of {L}yapunov equations}, SIAM J.
  Matrix Anal. Appl., 31 (2010), pp.~2553--2579.

\bibitem{Wen:2010}
{\sc Z.~Wen, W.~Yin, and Y.~Zhang}, {\em Solving a low-rank factorization model
  for matrix completion by a non-linear successive over-relaxation algorithm},
  Tech. Report TR10-07, CAAM Rice, 2010.

\end{thebibliography}
